\documentclass[11pt,reqno]{amsart}

\usepackage[T1]{fontenc}
\usepackage{mathptmx}
\usepackage{microtype}
\usepackage{amsmath,amssymb,amsthm,mathtools,mathrsfs}
\usepackage{enumitem}
\usepackage[colorlinks=true,linkcolor=blue,citecolor=blue,urlcolor=blue]{hyperref}
\usepackage[nameinlink,capitalise]{cleveref}
\usepackage{geometry}
\numberwithin{equation}{subsection}

\newtheorem{theorem}{Theorem}[subsection]
\newtheorem{proposition}[theorem]{Proposition}
\newtheorem{lemma}[theorem]{Lemma}
\newtheorem{corollary}[theorem]{Corollary}
\newtheorem{problem}[theorem]{Problem}
\crefname{problem}{problem}{problems}
\Crefname{problem}{Problem}{Problems}
\theoremstyle{definition}
\newtheorem{definition}[theorem]{Definition}
\theoremstyle{remark}
\newtheorem{remark}[theorem]{Remark}

\newcommand{\ddc}{\mathrm{dd}^{c}}

\newcommand{\bP}{\mathbb P}
\newcommand{\C}{\mathbb C}
\newcommand{\R}{\mathbb R}

\newcommand{\dist}{\operatorname{dist}}

\newcommand{\Xsurf}{\mathscr{X}}
\newcommand{\Wset}{\mathscr{W}}
\newcommand{\pmeas}{\sigma}
\newcommand{\diff}{\mathrm{d}}
\newcommand{\imag}{\mathrm{i}}
\newcommand{\e}{\mathrm{e}}
\newcommand{\geo}{\mathrm{g}}
\newcommand{\heo}{\mathrm{h}}
\newcommand{\calB}{\mathcal{B}}

\title[Exponential Green energies]{Exponential Green energies, quantization,
and minimizers on compact Riemann surfaces}
\date{}

\author{Haakan Hedenmalm}
\address{
Department of Mathematics and Computer Science,
St Petersburg State University, St Petersburg, Russia.
\\
Beijing Institute for Mathematical Sciences and Applications,
Huairou District, Beijing 101408, China
\\
Department of Mathematics and Statistics, University of Reading, U.K.
\\
Department of Mathematics, University of Aveiro, Portugal.
}
\email{haakan00@gmail.com}

\author{Song-Yan Xie}
\address{State Key Laboratory of Mathematical Sciences, Academy of Mathematics
and Systems Science, Chinese Academy of Sciences, Beijing 100190, China;
School of Mathematical Sciences, University of Chinese Academy of Sciences,
Beijing 100049, China}
\email{xiesongyan@amss.ac.cn}


\usepackage{graphicx}
\begin{document}

\begin{abstract}
Let $\Xsurf$ be a compact Riemann surface with a smooth positive probability form
$\pmeas$, and denote its normalized Green kernel by $U_{\Xsurf}$.  We study sequences of configurations
$\Wset_N\in \Xsurf^N$ for which the $\beta$-exponential Green energy
\[
 \calB_{\beta}(\Wset_N;\Xsurf)=\int_{\Xsurf}\exp\!\left(\beta\sum_{w\in \Wset_N}U_{\Xsurf}(z,w)\right)\diff\pmeas(z)
\]
is uniformly bounded at one fixed exponent $\beta>0$.  This hypothesis
yields uniform $L^1$ control of the mean-zero Green potentials and a
second-order discrepancy estimate of order $N^{-1}$.  By heat smoothing we
then obtain the optimal atomic transport scale
$W_1(\nu_{\Wset_N},\pmeas)\asymp N^{-1/2}$, sharp quadrature estimates for every
concave modulus of continuity, mesoscopic counting and heat-regularization
bounds, and negative Sobolev estimates.  We also analyze minimizers of the
$\beta$-exponential Green energy.  For each fixed $\beta$ they exist, have no collisions,
satisfy the optimal transport scale, and obey an intrinsic force-balance equation.
The dependence of the optimal constant on $\beta$ is a strictly increasing
homeomorphism whose one-sided logarithmic derivatives equal $\beta^{-2}$ times
the extremal Gibbs relative entropies of the active minimizers.

Finally, we study the optimal constants themselves.  They cannot collapse to the
Jensen value $1$: we prove $\Theta_{N,\beta}(\Xsurf)-1\geq c_{\Xsurf}\beta^2$ uniformly in
$N$, at the exact order in $\beta$.  On the round sphere, for
$N=2$ and every $\beta$ the energy is a ratio of gamma factors times a Gauss
hypergeometric function of the squared chordal distance to the antipode, the
minimizers are the antipodal pairs, and
$\Theta_{2,\beta}^{1/\beta}\to e/2$.  At the exponent $\beta=2$ and every $N$ the
energy equals $e^N/\bigl((N+1)N!\bigr)$ times the permanent of the Gram matrix of
the configuration; for $N=3$ this gives $\Theta_{3,2}=e^3/16$, with the
equilateral great-circle triangles as the unique minimizers.  A companion
Green-packing theorem supplies uniformly bounded-energy configurations for every
$N$, so the hypothesis is non-vacuous.
\end{abstract}

\maketitle

\tableofcontents

\section{Introduction}\label{sec:intro}

\subsection{The bounded \texorpdfstring{$\beta$}{beta}-exponential Green-energy class}
\label{subsec:energy-class}

Let $\Xsurf$ be a connected compact Riemann surface equipped with a smooth positive
$(1,1)$-form $\pmeas$ satisfying $\int_{\Xsurf}\diff\pmeas=1$.  We use $\pmeas$ for both the form and
its measure.  Throughout we use the conventions
\[
 d^c=\frac{\imag}{2\pi}(\bar\partial-\partial),
 \qquad
 \ddc=\frac{\imag}{\pi}\partial\bar\partial,
\]
so that $\ddc\log|z|=\delta_0$ on $\C$, and the normalized Green kernel is the
unique symmetric function $U_{\Xsurf}(z,w)$ with
\begin{equation}\label{eq:green-normalization}
 \ddc_zU_{\Xsurf}(z,w)=\delta_w-\pmeas,
 \qquad
 \int_{\Xsurf}U_{\Xsurf}(z,w)\,\diff\pmeas(z)=0 .
\end{equation}
It is smooth off the diagonal and, in a local coordinate at $w$,
\begin{equation}\label{eq:green-local}
 U_{\Xsurf}(z,w)=\log|z-w|+O(1);
\end{equation}
the joint form of this expansion is recorded in
\Cref{lem:green-joint-local}.  For the conformal metric whose area form is
$\pmeas$, our Laplace--Beltrami operator $\Delta$ is nonnegative and
$\ddc f=-(2\pi)^{-1}(\Delta f)\pmeas$.  We fix once and for all a smooth
Riemannian metric $\geo$ on $\Xsurf$; all distances, Wasserstein distances, and
H\"older seminorms below are formed with $\geo$ unless stated otherwise, and any
other fixed metric changes only constants.  Constants may depend on the fixed
data $(\Xsurf,\pmeas,\geo)$ and on $\beta$, but never on $N$. For a smooth function
$\phi$, the norm of its two-form $\ddc\phi$ means
\[
 \|\ddc\phi\|_{\infty}
 :=\left\|\frac{\ddc\phi}{\pmeas}\right\|_{L^\infty(\Xsurf,\pmeas)}.
\]
This fixes the numerical coefficient in the discrepancy estimate below.

For an ordered configuration $\Wset=(w_1,\ldots,w_N)\in \Xsurf^N$, write
\begin{equation}\label{eq:potential-and-measure}
 V_{\Wset}:=\sum_{j=1}^NU_{\Xsurf}(\,\cdot\,,w_j),
 \qquad
 \nu_{\Wset}:=\frac1N\sum_{j=1}^N\delta_{w_j},
\end{equation}
and, for $\beta>0$, call the following quantities the $\beta$-exponential Green energy
and its optimal value:
\begin{equation}\label{eq:B-def}
 \calB_{\beta}(\Wset;\Xsurf):=\int_{\Xsurf}e^{\beta V_{\Wset}}\,\diff\pmeas,
 \qquad
 \Theta_{N,\beta}(\Xsurf):=\inf_{\Wset\in \Xsurf^N}\calB_{\beta}(\Wset;\Xsurf).
\end{equation}
Since $\int_{\Xsurf}V_{\Wset}\,\diff\pmeas=0$ and $\pmeas(\Xsurf)=1$, Jensen's inequality gives
$\calB_{\beta}(\Wset;\Xsurf)\geq1$ always.  The singularities of $V_{\Wset}$ are negative, so the
integrand vanishes rather than blows up at the points $w_j$.

The $\beta$-exponential Green energy \eqref{eq:B-def} arises in the study of
Bombieri-type inequalities and of equidistribution of points.  The sphere and
the torus, including the uniform packing bounds and the explicit
sigma-function products in genus one, were studied by
Etayo, Hedenmalm, and Ortega-Cerd\`a
\cite[Sections~2.1--2.3]{EtayoHedenmalmOrtega}.  The point of view taken here is that the
consequences of a bounded exponential moment can be separated entirely from any
such explicit representation.

This paper is organized around a single hypothesis.

\begin{definition}[Bounded $\beta$-exponential Green energy]\label{def:energy-class}
Let $\beta>0$ and $K<\infty$.  A sequence of configurations $\Wset_N\in \Xsurf^N$,
$N\geq1$, belongs to the class $\calB_{\beta}(K)$ if
\begin{equation}\label{eq:energy-hypothesis}
 \calB_{\beta}(\Wset_N;\Xsurf)\leq K
 \qquad\text{for every }N .
\end{equation}
\end{definition}

Every statement in \Cref{sec:discrepancy,sec:quantization,sec:mesoscopic} is
proved under \eqref{eq:energy-hypothesis} alone.  In particular, no
holomorphic, algebraic, or spectral origin for the configurations is assumed:
the hypothesis is a scalar integral bound, and it is the only assumption.  Proceeding
this way has two benefits.  It isolates exactly how much of the fine
distributional behavior of an $N$-point configuration is forced by one
exponential moment, and it makes the resulting theory applicable to any family
satisfying \eqref{eq:energy-hypothesis}, however obtained.

The hypothesis is non-vacuous, and indeed uniformly attainable, by the
companion Green-packing theorem \cite[Theorem~1.4.1]{PaperA}.

\begin{theorem}[Attainability]\label{thm:packing-input}
For every compact Riemann surface $(\Xsurf,\pmeas)$ there is a constant
$C_{\Xsurf,\pmeas}<\infty$ and, for every $N\geq1$, a configuration $\Wset_N\in \Xsurf^N$ with
\[
 \sup_{z\in \Xsurf}V_{\Wset_N}(z)\leq C_{\Xsurf,\pmeas}.
\]
Consequently $\Wset_N\in\calB_{\beta}(e^{\beta C_{\Xsurf,\pmeas}})$ simultaneously for
every $\beta>0$, and
$1\leq\Theta_{N,\beta}(\Xsurf)\leq e^{\beta C_{\Xsurf,\pmeas}}$.
\end{theorem}

\Cref{thm:packing-input} guarantees that the classes $\calB_{\beta}(K)$ are
nonempty and that $\Theta_{N,\beta}(\Xsurf)$ is bounded; see
\Cref{cor:minimizer-in-class}.  The upper small-exponent estimate in
\Cref{thm:theta-uniform-noncollapse}\textup{(iv)} also uses the local norm model
of the holomorphic construction in \cite[Proposition~6.1.2 and Theorem~7.1.1]{PaperA};
the required $L^2$ estimate is justified in its proof below.  The consequences
of \eqref{eq:energy-hypothesis} in
\Cref{sec:discrepancy,sec:quantization,sec:mesoscopic} do not depend on how the
configurations were produced.  Two consequences of the pointwise bound of
\Cref{thm:packing-input} are stronger than what \eqref{eq:energy-hypothesis}
gives, and we flag the distinction wherever it matters --- see
\Cref{rem:sharp-and-nonsharp}.

\subsection{Notation for discrepancy}\label{subsec:discrepancy-notation}

For $0<\alpha\leq1$ set
\[
 [f]_{C^{0,\alpha}(\geo)}
 :=\sup_{x\neq y}\frac{|f(x)-f(y)|}{d_{\geo}(x,y)^\alpha},
 \qquad
 \mathcal D_{\alpha}(\nu,\pmeas)
 :=\sup_{\substack{f\in C^{0,\alpha}(\Xsurf;\R)\\ [f]_{C^{0,\alpha}(\geo)}\leq1}}
   \left|\int_{\Xsurf}f\,\diff(\nu-\pmeas)\right| .
\]
More generally, for a continuous nondecreasing concave modulus
$\omega:[0,\infty)\to[0,\infty)$ with $\omega(0)=0$ and $\omega(t)>0$ for
$t>0$, set
\[
 \mathcal D_{\omega}(\nu,\pmeas)
 :=\sup_{\substack{f\in C(\Xsurf;\R)\\
        |f(x)-f(y)|\leq\omega(d_{\geo}(x,y))\ \forall x,y}}
   \left|\int_{\Xsurf}f\,\diff(\nu-\pmeas)\right| .
\]
Finally, the covering radius of an ordered configuration is
\[
 \mathcal R_{\geo}(\Wset):=\sup_{x\in \Xsurf}\min_{1\leq j\leq N}d_{\geo}(x,w_j).
\]

\subsection{The four headline results}\label{subsec:headline}

\emph{Energy to discrepancy.}  A single exponential moment controls the
positive excursions of $V_{\Wset_N}$ with an exponential spatial tail, hence
controls $\|V_{\Wset_N}\|_{L^1}$; since $\ddc V_{\Wset_N}=N(\nu_{\Wset_N}-\pmeas)$,
integration by parts converts this into a second-order discrepancy estimate of
order $N^{-1}$.  This is \Cref{prop:energy-to-discrepancy}, and it is the
foundational statement for the discrepancy and mesoscopic estimates below,
which follow by choosing suitable test functions.

\emph{Sharp metric quantization.}  Heat smoothing turns the second-order
estimate into the optimal Lipschitz-dual rate, and the atomic lower bound of
\Cref{prop:metric-atomic-transport} matches it:
\[
 c_{\Xsurf}N^{-1/2}\leq W_1(\nu_{\Wset_N},\pmeas)\leq C_{\Xsurf,K,\beta}N^{-1/2}.
\]
More generally, for every concave modulus $\omega$ the two-sided scale is
$\omega(N^{-1/2})$, with constants independent of $\omega$
(\Cref{thm:energy-concave-modulus}).  The exponent $1/2$ is the real-dimensional
quantization exponent of a surface, in agreement with the classical quantization
theory of probability distributions \cite{GrafLuschgy}.
\Cref{prop:abstract-modulus-quantization} isolates the metric-measure mechanism
responsible: on a compact metric probability space with
$\mathfrak m(B(x,r))\leq Cr^m$, the lower quantization scale is
$\omega(N^{-1/m})$, and configurations satisfying
$W_1(\nu_N,\mathfrak m)\leq A N^{-1/m}$ attain the matching upper bound.

\emph{Mesoscopic consequences.}  Testing against scaled bumps and against the
heat kernel gives two-sided point counts above the scale $N^{-1/4}$, uniform
heat regularization, a covering-radius bound, multiplicity bounds, and
negative Sobolev convergence; see \Cref{sec:mesoscopic}.

\emph{Minimizers.}  For each fixed $\beta>0$ and $N$, minimizers of
$\calB_{\beta}(\,\cdot\,;\Xsurf)$ exist, consist of $N$ distinct points, satisfy the
optimal $W_1$ scale, and obey the intrinsic Euler--Lagrange balance equation
\eqref{eq:Euler-Lagrange-minimizer}.  The exponent dependence is described
exactly: $\beta\mapsto\Theta_{N,\beta}^{1/\beta}$ is a strictly increasing
homeomorphism of $(0,\infty)$ onto $(1,\Theta_{N,\infty})$ whose one-sided
logarithmic derivatives equal $\beta^{-2}$ times the extremal Gibbs relative
entropies of the active minimizers (\Cref{prop:theta-Gibbs-envelope}).  See \Cref{sec:minimizers}.

\subsection{Sharp and nonsharp scales}\label{subsec:sharpness}

Because several different exponents appear, we state at the outset which are
optimal and which are not.

\begin{remark}[What is sharp under \eqref{eq:energy-hypothesis}]
\label{rem:sharp-and-nonsharp}
\begin{enumerate}[label=\textup{(\roman*)}]
\item The $W_1$ scale $N^{-1/2}$ and the concave-modulus scale
$\omega(N^{-1/2})$ are \emph{sharp} for $N$-atomic measures: the lower bounds
in \Cref{cor:Wasserstein-lower} and \Cref{thm:energy-concave-modulus} hold for
every probability measure supported on at most $N$ points.
\item The general $W_p$ bound $O(N^{-1/(2p)})$ of
\Cref{cor:energy-Wasserstein-p}, for $p>1$, is obtained from the macroscopic
$W_1$ estimate and the diameter of $\Xsurf$ alone, and is \emph{not} claimed to be
sharp.  For the spectral configurations of \cite{PaperA} the sharp order
$N^{-1/2}$ does hold simultaneously for all $1\leq p\leq\infty$, because the
construction there couples each cell of an area-regular partition to its own
node and so retains transport information at the microscopic scale. This
additional coupling is not furnished by the energy-to-$W_1$ argument above.
\item The energy-only covering radius $O(N^{-1/4})$ of
\Cref{cor:energy-covering-radius} is \emph{weaker} than the $O(N^{-1/2})$
covering enjoyed by both families of \cite{PaperA}.  It should be read as a
consequence of the weak hypothesis, not as the geometric scale of any particular
configuration.  The threshold $N^{-1/4}$ is exactly where the area
term $Nr^2$ and the discrepancy error $r^{-2}$ balance.
\item Optimal $W_1$ transport does \emph{not} imply microscopic separation:
transport distance is insensitive to a small number of tightly clustered pairs.
The correct energy-only substitute is the multiplicity bound
$O(\sqrt N)$ of \Cref{cor:energy-multiplicity}, and the transport-only
collision exponents are genuinely attained, by
\Cref{prop:flat-torus-multiplicity-sharp}.
\end{enumerate}
\end{remark}

\subsection{Organization}\label{subsec:organization}

\Cref{sec:green-preliminaries} records the joint local form of the Green
kernel.  \Cref{sec:discrepancy} proves the energy-to-discrepancy principle.
\Cref{sec:quantization} develops heat smoothing, the sharp $W_1$ scale,
atomic lower bounds, and concave-modulus quantization.
\Cref{sec:mesoscopic} treats mesoscopic counting, heat regularization,
multiplicity, and negative Sobolev estimates.  In
\Cref{sec:optimal-constants} we study the dependence of the optimal
energies on $\beta$ and present the exact spherical models for two and
three points, along with higher-cardinality examples and open problems.
\Cref{sec:minimizers} establishes existence, absence of collisions,
equidistribution, and the intrinsic balance law for minimizers.
Finally, \Cref{sec:microscopic} explains which microscopic questions
remain open.

\begin{figure}[tbp]
\centering
\includegraphics[width=0.76\textwidth]{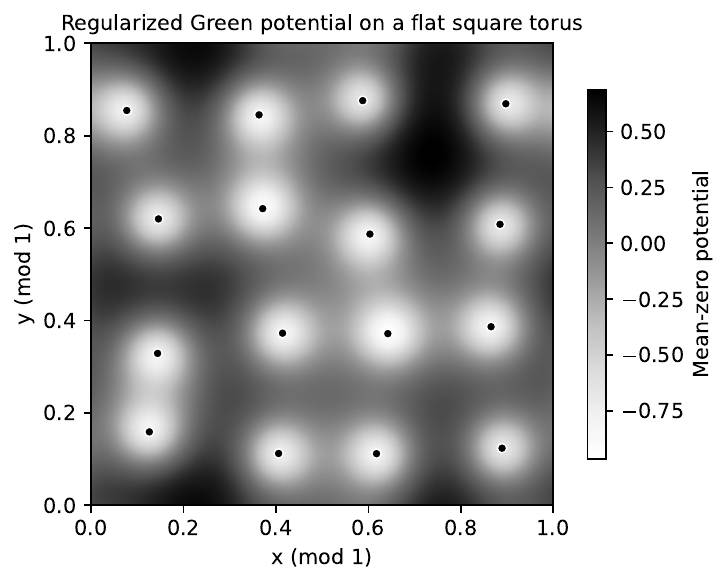}
\caption{A computed, heat-regularized Green potential on the flat square torus, for sixteen points shown as black dots. The potential is normalized to have mean zero. The grayscale depicts the regularized potential, rather than the singular Green potential itself. In our sign convention the unregularized potential tends to $-\infty$ at each configuration point.}
\label{fig:green-potential-illustration}
\end{figure}

\section{The Green kernel near the diagonal}\label{sec:green-preliminaries}

\subsection{Joint local regularity}

The variational arguments of \Cref{sec:minimizers} move the poles of the Green
potential, so we record the standard parameter-dependent form of
\eqref{eq:green-local} with a regular part that is smooth in \emph{both}
variables.

\begin{lemma}[Joint local form of the Green kernel]
\label{lem:green-joint-local}
Let $\xi:\Omega\to D\subset\C$ be a holomorphic coordinate and let
$D'\Subset D$ be a disk.  There is a function
$R_{\xi}\in C^\infty(D'\times D')$ such that
\begin{equation}\label{eq:green-joint-local}
 U_{\Xsurf}\bigl(\xi^{-1}(x),\xi^{-1}(q)\bigr)
 =\log|x-q|+R_{\xi}(x,q),
 \qquad x,q\in D',\quad x\ne q.
\end{equation}
For a fixed finite coordinate atlas, all finite-order norms of the
corresponding regular parts are uniformly bounded after passing to
relatively compact subcharts.  Consequently, for every $\gamma>0$,
$e^{\gamma U_{\Xsurf}(z,w)}$ extends continuously by zero across the diagonal.
Moreover, symmetry gives
\begin{equation}\label{eq:green-pole-variable}
 \ddc_wU_{\Xsurf}(z,w)=\delta_z-\pmeas
\end{equation}
in the pole variable.
\end{lemma}

\begin{proof}
Subtracting the logarithm and adding a local potential for $\pmeas$ makes the
result harmonic away from the pole and bounded across it, hence harmonic throughout; its
boundary values then depend smoothly on the pole, and elliptic regularity with
parameters gives joint smoothness.

Choose disks $D'\Subset D''\Subset D$ and a smooth local potential
$\psi$ on $\xi^{-1}(D'')$ with $\ddc\psi=\pmeas$.  For $q\in D'$, the
distribution
\[
 H_q(x):=U_{\Xsurf}\bigl(\xi^{-1}(x),\xi^{-1}(q)\bigr)
          -\log|x-q|+\psi\bigl(\xi^{-1}(x)\bigr)
\]
is harmonic on $D''\setminus\{q\}$ and bounded near $q$ by
\eqref{eq:green-local}; the singularity is removable.  On
$\partial D''$ the variables stay uniformly separated, so the boundary
values of $H_q$ depend smoothly on $q\in D'$.  The Poisson solution
operator, or equivalently interior elliptic regularity with parameters,
therefore makes $(x,q)\mapsto H_q(x)$ smooth on $D'\times D'$.  Taking
$R_{\xi}(x,q)=H_q(x)-\psi(\xi^{-1}(x))$ proves
\eqref{eq:green-joint-local}.  Compactness gives the finite-atlas bounds.
The continuous extension follows from
$e^{\gamma U_{\Xsurf}}=|x-q|^\gamma e^{\gamma R_{\xi}}$, and
\eqref{eq:green-pole-variable} follows from symmetry and
\eqref{eq:green-normalization}.
\end{proof}

\section{Energy to discrepancy}\label{sec:discrepancy}

\subsection{Exponential moments and second-order discrepancy}

The main consequences rest on the following principle: one bounded exponential
moment forces an exponential spatial tail for the positive part of the
potential, hence an $L^1$ bound, hence --- through
$\ddc V_{\Wset}=N(\nu_{\Wset}-\pmeas)$ --- a second-order discrepancy estimate of order
$N^{-1}$.  No structure beyond \eqref{eq:energy-hypothesis} is used.

\begin{proposition}[Energy-to-discrepancy]
\label{prop:energy-to-discrepancy}
For $\Wset_N=(w_{1,N},\ldots,w_{N,N})\in \Xsurf^N$, write
\[
 V_{\Wset_N}(z)=\sum_{j=1}^NU_{\Xsurf}(z,w_{j,N}),
 \qquad \nu_{\Wset_N}=\frac1N\sum_{j=1}^N\delta_{w_{j,N}}.
\]
Suppose that $\calB_{\beta}(\Wset_N;\Xsurf)\leq K$.  Then $K\geq1$.  Write
\[
 a_K:=\frac{\log K}{\beta}.
\]
For every $t\geq0$,
\begin{equation}\label{eq:energy-positive-tail}
 \pmeas\{z\in \Xsurf:V_{\Wset_N}(z)>a_K+t\}\leq e^{-\beta t}.
\end{equation}
Consequently, for every $1\leq q<\infty$,
\begin{equation}\label{eq:energy-positive-moments}
 \|(V_{\Wset_N})_+\|_{L^q(\Xsurf,\pmeas)}
 \leq\frac{\log K+\Gamma(q+1)^{1/q}}{\beta}.
\end{equation}
More precisely, if $0<\alpha<\beta$, then
\begin{equation}\label{eq:energy-shifted-exponential-moment}
 \int_{\Xsurf}\exp\!\left(
   \alpha\bigl((V_{\Wset_N})_+-a_K\bigr)_+
 \right)\diff\pmeas
 \leq\frac{\beta}{\beta-\alpha}.
\end{equation}
Finally,
\begin{equation}\label{eq:energy-L1-control}
 \|V_{\Wset_N}\|_{L^1(\Xsurf,\pmeas)}
 \leq\frac{2(1+\log K)}{\beta},
\end{equation}
and consequently
\begin{equation}\label{eq:energy-discrepancy}
 \left|\int_{\Xsurf}\phi\,\diff(\nu_{\Wset_N}-\pmeas)\right|
 \leq\frac{2(1+\log K)}{\beta N}\|\ddc\phi\|_{\infty}
 \qquad(\phi\in C^\infty(\Xsurf)).
\end{equation}
\end{proposition}

\begin{proof}
The potential is locally integrable: near a node it is a finite sum of
logarithms plus a smooth function.  Each Green kernel has mean zero, hence
$\int_{\Xsurf}V_{\Wset_N}\,\diff\pmeas=0$.  Since $\pmeas(\Xsurf)=1$, Jensen's inequality gives
\[
 1=\exp\!\left(\beta\int_{\Xsurf}V_{\Wset_N}\,\diff\pmeas\right)
 \leq\int_{\Xsurf}e^{\beta V_{\Wset_N}}\diff\pmeas\leq K,
\]
so $K\geq1$ and $a_K\geq0$.  Markov's inequality gives
\[
 \begin{aligned}
 \pmeas\{V_{\Wset_N}>a_K+t\}
 &\leq e^{-\beta(a_K+t)}\int_{\Xsurf}e^{\beta V_{\Wset_N}}\diff\pmeas \\
 &\leq K e^{-\beta a_K}e^{-\beta t}=e^{-\beta t},
 \end{aligned}
\]
which is \eqref{eq:energy-positive-tail}.

Write $Y=((V_{\Wset_N})_+-a_K)_+$.  Then
$\pmeas\{Y>t\}\leq e^{-\beta t}$ for $t\geq0$.  The layer-cake formula and
Tonelli's theorem yield
\[
 \|Y\|_{L^q}^q
 =q\int_0^\infty t^{q-1}\pmeas\{Y>t\}\,dt
 \leq q\int_0^\infty t^{q-1}e^{-\beta t}\,dt
 =\frac{\Gamma(q+1)}{\beta^q}.
\]
Since $(V_{\Wset_N})_+\leq a_K+Y$, Minkowski's inequality proves
\eqref{eq:energy-positive-moments}.  The same tail integration gives, for
$0<\alpha<\beta$,
\[
 \begin{aligned}
 \int_{\Xsurf}e^{\alpha Y}\diff\pmeas
 &=1+\alpha\int_0^\infty e^{\alpha t}\pmeas\{Y>t\}\,dt \\
 &\leq1+\alpha\int_0^\infty e^{-(\beta-\alpha)t}\,dt
 =\frac{\beta}{\beta-\alpha},
 \end{aligned}
\]
proving \eqref{eq:energy-shifted-exponential-moment}.

Taking $q=1$ in \eqref{eq:energy-positive-moments} gives
\[
 \int_{\Xsurf}(V_{\Wset_N})_+\diff\pmeas\leq\frac{1+\log K}{\beta}.
\]
The mean-zero identity makes the integrals of the positive and negative
parts equal.  This proves \eqref{eq:energy-L1-control}.  Since
$\ddc V_{\Wset_N}=N(\nu_{\Wset_N}-\pmeas)$, distributional integration by parts gives
\eqref{eq:energy-discrepancy}.
\end{proof}

\begin{corollary}[Energy excludes large holes]
\label{cor:energy-covering-radius}
Let $\Wset_N\in \Xsurf^N$ satisfy $\calB_{\beta}(\Wset_N;\Xsurf)\leq K$, where $K$ and
$\beta>0$ are independent of $N$.  For any fixed Riemannian distance $d$
on $\Xsurf$, its covering radius
\[
 \mathcal R(\Wset_N):=\sup_{x\in \Xsurf}\min_{w\in \Wset_N}d(x,w)
\]
satisfies, for all sufficiently large $N$,
\begin{equation}\label{eq:energy-covering-radius}
 \mathcal R(\Wset_N)\leq
 C_{\Xsurf}\left(\frac{1+\log K}{\beta N}\right)^{1/4}.
\end{equation}
\end{corollary}

\begin{proof}
Test the discrepancy estimate against a bump supported in a hole.  The bump has
$\pmeas$-mass of order $r^2$ but zero $\nu_{\Wset_N}$-mass, while its $\ddc$ norm is
$O(r^{-2})$; balancing these against \eqref{eq:energy-discrepancy} bounds $r$.

Let $r=\mathcal R(\Wset_N)$ and choose $x\in \Xsurf$ with
$d(x,\Wset_N)=r$.  There are uniform constants $r_*,c,C>0$ such that for every
$0<\rho\leq r_*$ one can choose a smooth cutoff $\phi_{x,\rho}$ satisfying
\[
 0\leq\phi_{x,\rho}\leq1,\qquad
 \phi_{x,\rho}=1\text{ on }B(x,\rho/4),\qquad
 \operatorname{supp}\phi_{x,\rho}\subset B(x,\rho/2),
\]
and
\begin{equation}\label{eq:covering-bump-estimates}
 \int_{\Xsurf}\phi_{x,\rho}\,\diff\pmeas\geq c\rho^2,
 \qquad
 \|\ddc\phi_{x,\rho}\|_{\infty}\leq C\rho^{-2}.
\end{equation}
These are the standard scaled bump estimates in a finite atlas.

First, $r<r_*$ for all sufficiently large $N$.  Otherwise use
$\rho=r_*$ below: its supporting ball contains no node, while
\eqref{eq:energy-discrepancy} makes the fixed positive number on the left
at most $C/N$, a contradiction.  We may therefore set $\rho=r$.  Since the
support of $\phi_{x,r}$ contains no node,
\[
 \int_{\Xsurf}\phi_{x,r}\,\diff\nu_{\Wset_N}=0.
\]
Combining \eqref{eq:covering-bump-estimates} with
\eqref{eq:energy-discrepancy} gives
\[
 cr^2
 \leq\left|\int_{\Xsurf}\phi_{x,r}\,\diff(\nu_{\Wset_N}-\pmeas)\right|
 \leq\frac{2(1+\log K)}{\beta N}\,Cr^{-2}.
\]
Thus $r^4\leq C_{\Xsurf}(1+\log K)/(\beta N)$, which is
\eqref{eq:energy-covering-radius}.
\end{proof}

\section{Sharp metric quantization}\label{sec:quantization}

\subsection{Transport bounds and atomic quantization}

Heat smoothing converts the second-order estimate of
\Cref{prop:energy-to-discrepancy} into the optimal Lipschitz-dual rate.  The
matching lower bound is a purely metric-measure fact about atomic measures, and
we prove it in that generality: this makes visible that the exponent $1/2$ comes
from the real dimension of the surface and not from the Green kernel.

\begin{corollary}[Optimal-scale transport bound]\label{cor:Wasserstein}
For every family of configurations $\Wset_N=(w_{1,N},\ldots,w_{N,N})$
satisfying
$K:=\sup_N \calB_{\beta}(\Wset_N;\Xsurf)<\infty$---in particular, for the
holomorphic divisors and the spectral-design configurations---let
$\nu_N=N^{-1}\sum_{j=1}^N\delta_{w_{j,N}}$.  Then
\begin{equation}\label{eq:Wasserstein-bound}
 W_1(\nu_N,\pmeas)\leq C_{\Xsurf,K,\beta}N^{-1/2},
\end{equation}
where $W_1$ is computed using any fixed Riemannian distance on $\Xsurf$.
\end{corollary}

\begin{proof}
The obstruction to applying \eqref{eq:energy-discrepancy} directly is that a
one-Lipschitz test function need not have bounded $\ddc$.  Heat smoothing repairs
this: it costs $O(\sqrt t)$ in sup norm and gains $O(t^{-1/2})$ in $\ddc$, and
balancing the two at $t=N^{-1}$ gives the optimal rate.

By Kantorovich--Rubinstein duality \cite[Ch.~5]{Villani} it is enough first to
test against real
functions $\phi$ with Lipschitz constant at most one for a fixed conformal
metric $\geo_{\pmeas}$ whose area form is $\pmeas$.  Let
$\phi_t=e^{-t\Delta_{\pmeas}}\phi$, for the nonnegative Laplace--Beltrami
operator of $\geo_{\pmeas}$.  Its heat kernel satisfies
\[
 |\nabla_x^k H_t(x,y)|\leq C_k t^{-1-k/2}
 \exp\!\left(-\frac{d(x,y)^2}{C_k t}\right),\qquad k=0,2,\quad 0<t\leq1.
\]
Indeed, near the diagonal this follows by differentiating the local expansion
in \cite[Theorem~1.1]{Ludewig}; away from the diagonal, the extra inverse
powers of $t$ in \cite[Theorem~3.5]{Ludewig} are absorbed by the Gaussian.
Consequently, for $0<t\leq1$,
\begin{equation}\label{eq:heat-smoothing-Lipschitz}
 \|\phi-\phi_t\|_{L^\infty}\leq C\sqrt t,
 \qquad
 \|\ddc\phi_t\|_{L^\infty}\leq Ct^{-1/2}.
\end{equation}
The first estimate follows by averaging the Lipschitz bound against the
heat kernel.  For the second, differentiate the heat kernel twice and use
its cancellation against constants; one derivative is absorbed by the
Lipschitz regularity of $\phi$.
More explicitly, if $H_t(x,y)$ is the heat kernel, then
\[
 \phi_t(x)-\phi(x)=\int_{\Xsurf}H_t(x,y)(\phi(y)-\phi(x))\,\diff\pmeas(y).
\]
The Gaussian first-moment bound for $H_t$ is $O(\sqrt t)$, proving the
first estimate.  Since $\int_{\Xsurf}\nabla_x^2H_t(x,y)\,\diff\pmeas(y)=0$,
\[
 \nabla^2\phi_t(x)=
 \int_{\Xsurf}\nabla_x^2H_t(x,y)(\phi(y)-\phi(x))\,\diff\pmeas(y).
\]
The derivative Gaussian bound and
$|\phi(y)-\phi(x)|\leq d(x,y)$ make the last integral $O(t^{-1/2})$.
On the fixed surface, $|\ddc\phi_t|$ is bounded by a constant multiple of
$|\nabla^2\phi_t|$, giving the second estimate in
\eqref{eq:heat-smoothing-Lipschitz}.

Apply \eqref{eq:energy-discrepancy} to $\phi_t$.  Since both
measures have mass one,
\[
 \left|\int_{\Xsurf}\phi\,\diff(\nu_N-\pmeas)\right|
 \leq2\|\phi-\phi_t\|_{\infty}
   +\frac{C}{N}\|\ddc\phi_t\|_{\infty}
 \leq C\left(\sqrt t+\frac{1}{N\sqrt t}\right).
\]
Taking $t=N^{-1}$ and then the supremum over all such $\phi$ proves
\eqref{eq:Wasserstein-bound} for the distance of $\geo_{\pmeas}$.  Any two fixed
Riemannian distances on the compact surface are bi-Lipschitz equivalent, so
the same order, with a changed constant, holds for the distance stated in
the corollary.
\end{proof}

The next proposition is the general atomic lower bound.  It is stated for an
arbitrary compact metric probability space because the same statement supplies
both the sharpness of the $W_1$ rate and, in
\Cref{cor:metric-multiplicity-profile}, the collision profile.

\begin{corollary}[Finite-order transport from energy alone]
\label{cor:energy-Wasserstein-p}
Under the hypotheses of \Cref{cor:Wasserstein}, for every fixed
$1\leq p<\infty$,
\begin{equation}\label{eq:energy-Wasserstein-p}
 W_p(\nu_N,\pmeas)\leq C_{\Xsurf,K,\beta,p}N^{-1/(2p)}.
\end{equation}
\end{corollary}

\begin{proof}
Fix the same Riemannian ground distance used for $W_1$, and let
$D=\operatorname{diam}(\Xsurf)$.  For every coupling $\pi$,
\[
 \int_{\Xsurf\times \Xsurf}d(x,y)^p\,\diff\pi
 \leq D^{p-1}\int_{\Xsurf\times \Xsurf}d(x,y)\,\diff\pi.
\]
Apply this to an optimal $W_1$ coupling (or to an approximating sequence of
couplings) and then use \eqref{eq:Wasserstein-bound}.  It follows that
\[
 W_p(\nu_N,\pmeas)^p
 \leq D^{p-1}W_1(\nu_N,\pmeas)
 \leq C_{\Xsurf,K,\beta,p}N^{-1/2}.
\]
Taking the $p$-th root (and renaming the fixed constant) gives
\eqref{eq:energy-Wasserstein-p}.
\end{proof}

\begin{proposition}[Metric-dimensional atomic transport]
\label{prop:metric-atomic-transport}
Let $(Y,d,\mathfrak m)$ be a compact metric probability space.  Suppose
that, for some $m>0$, $C_{\rm vol}>0$, and $r_0>0$,
\begin{equation}\label{eq:metric-small-ball-upper}
 \mathfrak m(B_d(y,r))\leq C_{\rm vol}r^m
 \qquad(y\in Y,\ 0<r\leq r_0).
\end{equation}
Then there is $c_*=c_*(m,C_{\rm vol},r_0)>0$ such that every probability
measure $\nu$ on $Y$ satisfies
\begin{equation}\label{eq:metric-atomic-transport}
 W_1^d(\nu,\mathfrak m)
 \geq c_*\sum_{x\in\operatorname{At}(\nu)}
              \nu(\{x\})^{1+1/m},
\end{equation}
where $\operatorname{At}(\nu)$ is the at most countable set of atoms of
$\nu$.
\end{proposition}

\begin{proof}
The point is that an atom of mass $a$ cannot be reached cheaply: the ambient
measure contains only $O(a)$ mass within distance $a^{1/m}$ of it, so at least half
the atom's mass must travel that far.  Summing over atoms, which occupy disjoint
target fibers, gives the claim.

Write $a_x=\nu(\{x\})$ and choose
\[
 c_0:=\min\{r_0,(2C_{\rm vol})^{-1/m}\},
 \qquad r_x:=c_0a_x^{1/m}.
\]
Then $r_x\leq r_0$, and \eqref{eq:metric-small-ball-upper} gives
\begin{equation}\label{eq:atom-half-mass-ball}
 \mathfrak m(B_d(x,r_x))
 \leq C_{\rm vol}c_0^ma_x\leq\frac{a_x}{2}.
\end{equation}
Let $\pi\in\Pi(\mathfrak m,\nu)$ be any coupling, with first marginal
$\mathfrak m$ and second marginal $\nu$.  Since
$\pi(Y\times\{x\})=a_x$ and
$\pi(B_d(x,r_x)\times\{x\})\leq\mathfrak m(B_d(x,r_x))$,
\eqref{eq:atom-half-mass-ball} implies
\[
 \pi\bigl((Y\setminus B_d(x,r_x))\times\{x\}\bigr)
 \geq\frac{a_x}{2}.
\]
These subsets of $Y\times Y$ are pairwise disjoint as $x$ ranges over the
atoms.  Thus, for every finite $F\subset\operatorname{At}(\nu)$,
\[
 \int_{Y\times Y}d(y,z)\,\diff\pi(y,z)
 \geq\frac12\sum_{x\in F} a_xr_x
 =\frac{c_0}{2}\sum_{x\in F}a_x^{1+1/m}.
\]
Taking the supremum over all finite $F$ allows all atoms.  Taking the
infimum over $\pi$ then proves \eqref{eq:metric-atomic-transport} with
$c_*=c_0/2$.
\end{proof}

\begin{corollary}[Sharpness of the surface transport scale]
\label{cor:Wasserstein-lower}
There is $c_{\Xsurf}>0$ such that every probability measure supported on at most
$N$ points satisfies
\begin{equation}\label{eq:Wasserstein-lower}
 W_1(\nu,\pmeas)\geq c_{\Xsurf}N^{-1/2}.
\end{equation}
Consequently the rate in \eqref{eq:Wasserstein-bound} is optimal up to its
constant.
\end{corollary}

\begin{proof}
The fixed surface satisfies \eqref{eq:metric-small-ball-upper} with $m=2$.
If the nonzero masses of $\nu$ are $a_1,\ldots,a_M$, where $M\leq N$, then
convexity and $\sum_i a_i=1$ give
\[
 \sum_{i=1}^M a_i^{3/2}\geq M^{-1/2}\geq N^{-1/2}.
\]
Apply \eqref{eq:metric-atomic-transport}.
\end{proof}

\begin{corollary}[Atomic mass and the multiplicity moment]
\label{cor:atom-transport-lower}
There is $c_{\Xsurf}>0$ such that every probability measure $\nu$ on $\Xsurf$ and every
$x\in \Xsurf$ satisfy
\begin{equation}\label{eq:atom-transport-lower}
 W_1(\nu,\pmeas)\geq c_{\Xsurf}\nu(\{x\})^{3/2}.
\end{equation}
Moreover, if
$\nu_N=N^{-1}\sum_{j=1}^N\delta_{y_j}$ and the distinct nodes have
multiplicities $k_1,k_2,\ldots$, then
\begin{equation}\label{eq:transport-multiplicity-moment}
 W_1(\nu_N,\pmeas)
 \geq c_{\Xsurf}N^{-3/2}\sum_i k_i^{3/2}.
\end{equation}
\end{corollary}

\begin{proof}
Both assertions are the case $m=2$ of
\eqref{eq:metric-atomic-transport}; for the second one use
$\nu_N(\{x_i\})=k_i/N$.
\end{proof}

Combining the moment estimate with an optimal-order transport hypothesis gives a
full profile for the multiplicities: how many heavy sites there can be, how
much of the mass they can carry, and how large the largest can be.

\begin{corollary}[Multiplicity profile at the quantization scale]
\label{cor:metric-multiplicity-profile}
Assume the hypotheses of \Cref{prop:metric-atomic-transport}, and let
\[
 \nu_N=\frac1N\sum_{j=1}^N\delta_{y_j},
 \qquad A>0,\qquad
 W_1^d(\nu_N,\mathfrak m)\leq A N^{-1/m}.
\]
Write $x_1,\ldots,x_M$ for the distinct nodes and $k_i$ for the
multiplicity of $x_i$.  Set
\[
 q=1+\frac1m,\qquad C_A=\frac{A}{c_*},
 \qquad c_A=\left(\frac{c_*}{A}\right)^m,
\]
where $c_*$ is the structural constant in
\Cref{prop:metric-atomic-transport}.  Then
\begin{equation}\label{eq:metric-multiplicity-moment}
 \sum_{i=1}^M k_i^q\leq C_A N,
 \qquad
 M\geq c_A N.
\end{equation}
Moreover, for every real $T\geq1$,
\begin{align}
 \#\{i:k_i\geq T\}
 &\leq C_A N T^{-q},
 \label{eq:heavy-site-count}\\
 \frac1N\sum_{\{i:k_i\geq T\}}k_i
 &\leq C_A T^{-1/m}.
 \label{eq:heavy-entry-fraction}
\end{align}
If $k_{(1)}\geq k_{(2)}\geq\cdots$ is the decreasing rearrangement, then
\begin{equation}\label{eq:ordered-multiplicity-profile}
 k_{(j)}\leq C_A^{m/(m+1)}
              \left(\frac Nj\right)^{m/(m+1)}
 \qquad(1\leq j\leq M).
\end{equation}
In particular, under the surface energy hypothesis of
\Cref{cor:Wasserstein}, one has
\begin{equation}\label{eq:surface-energy-collision-profile}
 \#\operatorname{supp}\nu_N\geq c_{\Xsurf,K,\beta}N,
 \qquad
 \frac1N\sum_{\{i:k_i\geq T\}}k_i
 \leq C_{\Xsurf,K,\beta}T^{-1/2}.
\end{equation}
\end{corollary}

\begin{proof}
The moment bound comes from feeding the transport hypothesis into the atomic
lower bound.  Convexity then converts it into the count of distinct sites, and
Chebyshev-type truncation into the two tail estimates and the ordered profile.

By \eqref{eq:metric-atomic-transport},
\[
 c_*N^{-q}\sum_i k_i^q
 \leq W_1^d(\nu_N,\mathfrak m)
 \leq A N^{-1/m}.
\]
Since $q=1+1/m$, this gives
$\sum_i k_i^q\leq(A/c_*)N$.  Convexity and $\sum_i k_i=N$ give
\[
 \sum_{i=1}^M k_i^q
 \geq M\left(\frac NM\right)^q
 =N^qM^{1-q}.
\]
Comparison with the upper bound and $q-1=1/m$ yields
$M\geq(c_*/A)^mN$.

Every index with $k_i\geq T$ contributes at least $T^q$ to the moment,
which proves \eqref{eq:heavy-site-count}.  On the same set,
$k_i\leq T^{1-q}k_i^q$; summation and division by $N$ prove
\eqref{eq:heavy-entry-fraction}.  Finally,
$j k_{(j)}^q\leq\sum_i k_i^q$ gives
\eqref{eq:ordered-multiplicity-profile}.  The surface statement follows
from \eqref{eq:Wasserstein-bound} with $m=2$.
\end{proof}

The transport-only collision exponents just obtained are attained, so no
argument using $W_1$ alone can improve them.  The following construction is
\emph{not} a bounded-energy family; it shows the limits of the hypothesis
$W_1=O(N^{-1/m})$ in isolation, and for $m=2$ its multiplicity $N^{2/3}$ is
excluded by the stronger energy estimate of
\Cref{cor:energy-multiplicity} below.

\begin{proposition}[Sharpness of the multiplicity profile on flat tori]
\label{prop:flat-torus-multiplicity-sharp}
Let $m\geq1$ be an integer, and give
$\mathbb T^m=\mathbb R^m/\mathbb Z^m$ its flat distance and normalized
Lebesgue measure $\mathfrak m$.  Along the subsequence
\[
 N=r^{m(m+1)},\qquad r=2,3,\ldots,
\]
there are equal-weight empirical measures
$\nu_N=N^{-1}\sum_{j=1}^N\delta_{y_j}$ such that
\begin{equation}\label{eq:flat-torus-sharp-transport}
 c_mN^{-1/m}\leq W_1(\nu_N,\mathfrak m)
 \leq C_mN^{-1/m}.
\end{equation}
Writing $q=1+1/m$, their distinct-site multiplicities consist of one value
\[
 K=N^{m/(m+1)}
\]
and $N-K$ values equal to one.  Consequently,
\begin{equation}\label{eq:flat-torus-sharp-moment}
 \sum_i k_i^q=2N-K,\qquad
 \#\operatorname{supp}\nu_N=N-K+1.
\end{equation}
At the threshold $T=K$,
\begin{equation}\label{eq:flat-torus-sharp-tails}
 \#\{i:k_i\geq T\}=1=NT^{-q},\qquad
 \frac1N\sum_{\{i:k_i\geq T\}}k_i
 =\frac KN=T^{-1/m}.
\end{equation}
Thus the moment order, the ordered-profile bound at rank $j=1$ (hence the
maximal-multiplicity exponent), and both tail powers in
\Cref{cor:metric-multiplicity-profile} are sharp under an optimal-order
$W_1$ hypothesis alone.
\end{proposition}

\begin{proof}
We construct the configuration explicitly: partition the torus into $N$ equal
cubes, collapse a block of $K$ of them onto a single point, and send every
remaining cube to its own center.  The block is what creates the heavy atom, and
its side length is chosen so that the transport cost stays at the optimal order
while the atom is as large as that order permits.

Write $n=r^{m+1}$, so $n^m=N$, and partition the standard fundamental cube
$[0,1)^m$ into the half-open cells
\[
 C_a=\prod_{\ell=1}^m
 \left[\frac{a_{\ell}}{n},\frac{a_{\ell}+1}{n}\right),
 \qquad a_{\ell}\in\{0,\ldots,n-1\}.
\]
Their boundaries are null and may be assigned arbitrarily.  The block
\[
 Q=\prod_{\ell=1}^m
 \left[\frac1n,\frac{r^m+1}{n}\right)
\]
is the union of
$K=(r^m)^m=r^{m^2}=N^{m/(m+1)}$ cells and has side length $r^{-1}$.
Since $r\geq2$, its closure lies in $(0,1)^m$, so no wrap-around ambiguity
occurs.  Let $x_Q$ be the Euclidean center of $Q$.  Send the mass of every
cell contained in $Q$ to $x_Q$, and send every other cell to its own
center $c_a$.  Equivalently,
\[
 \nu_N=\frac1N\left(K\delta_{x_Q}
       +\sum_{C_a\not\subset Q}\delta_{c_a}\right).
\]
No center outside $Q$ equals $x_Q$, so the stated multiplicities are exact.

Let $\mathcal T_N:\mathbb T^m\to\mathbb T^m$ be this cellwise transport map, so
that the coupling $(\operatorname{Id},\mathcal T_N)_\#\mathfrak m$ has second
marginal $\nu_N$.  Since the flat distance is no larger than Euclidean distance in
the chosen fundamental cube, cells outside $Q$ contribute at most
$\sqrt m/(2n)$ to the transport cost, while the block contributes at most
\[
 \mathfrak m(Q)\frac{\sqrt m}{2r}
 =\frac{\sqrt m}{2}r^{-(m+1)}.
\]
Since $n^{-1}=r^{-(m+1)}=N^{-1/m}$, this proves the upper bound in
\eqref{eq:flat-torus-sharp-transport}.  Conversely, the atom at $x_Q$ has
mass
\[
 a=\frac KN=r^{-m}=N^{-1/(m+1)}.
\]
Applying \eqref{eq:metric-atomic-transport} to this atom gives
\[
 W_1(\nu_N,\mathfrak m)
 \geq c_ma^{1+1/m}
 =c_mr^{-(m+1)}=c_mN^{-1/m}.
\]
Finally, $K^q=N$, so
$\sum_i k_i^q=K^q+(N-K)=2N-K$.  The identities in
\eqref{eq:flat-torus-sharp-tails} follow from
$NK^{-q}=1$ and $K/N=K^{-1/m}$.
\end{proof}

\begin{remark}[Sharp exponent and the role of Green energy]
When $m$ is a positive integer, the power $1+1/m$ in
\eqref{eq:metric-atomic-transport} cannot be replaced by a smaller power of
the atomic mass under only an $m$-dimensional volume hypothesis.  Indeed,
on a flat $m$-torus, collapsing a ball of mass $a$ to its center has
transport cost $O(a^{1+1/m})$, while
\eqref{eq:metric-atomic-transport} supplies the matching lower order.  In
particular, the exponent is sharp for the surface application $m=2$.
On the surface, \Cref{cor:metric-multiplicity-profile} gives
$\sum_i k_i^{3/2}=O(N)$, linearly many distinct sites, and maximal
multiplicity $O(N^{2/3})$.
The next argument improves this to $O(\sqrt N)$ by using the stronger
second-order distributional estimate \eqref{eq:energy-discrepancy}, not
merely its $W_1$ consequence.
By itself the moment estimate does not imply bounded multiplicity or that all
but $o(N)$ listed entries are simple: a fixed positive density of double
sites is compatible with \eqref{eq:metric-multiplicity-moment}.
The flat-torus construction proves sharpness only under the optimal-order
$W_1$ hypothesis.  It is not a bounded-Green-energy construction.  For
$m=2$ its multiplicity $N^{2/3}$ is excluded by the stronger
energy-discrepancy bound $O(N^{1/2})$ proved below.
\end{remark}

We now replace the H\"older exponent by an arbitrary concave modulus.  This
subsumes the H\"older statement, so we state the power-modulus case first, prove
the general theorem, and then read the corollary off it.

\begin{corollary}[Sharp H\"older quadrature from energy alone]
\label{cor:energy-holder-discrepancy}
Under the hypotheses of \Cref{cor:Wasserstein}, for every
$0<\alpha\leq1$,
\begin{equation}\label{eq:energy-holder-discrepancy}
 c_{\Xsurf}N^{-\alpha/2}
 \leq\mathcal D_{\alpha}(\nu_N,\pmeas)
 \leq C_{\Xsurf,K,\beta}N^{-\alpha/2}.
\end{equation}
The constants may be chosen uniformly for $0<\alpha\leq1$.
\end{corollary}

\begin{proof}
This is \Cref{thm:energy-concave-modulus} specialized to $\omega(t)=t^\alpha$,
together with the observation that a snowflaked metric turns the H\"older
discrepancy into a transport distance.  Nothing in the proof of that theorem uses
the present corollary, so the order of presentation is not circular.

Take $d=d_{\geo}$, the fixed ground distance used in the definition of
$\mathcal D_{\alpha}$, and compute $W_1$ with this same distance.  Since
$d^\alpha$ is a metric,
Kantorovich--Rubinstein duality identifies $\mathcal D_{\alpha}(\nu_N,\pmeas)$
with the $W_1$ transport cost for the snowflaked distance $d^\alpha$.
Let $\pi$ be an optimal coupling for the original distance $d$.  Concavity
of $t\mapsto t^\alpha$ and Jensen's inequality give
\[
 \begin{aligned}
 \mathcal D_{\alpha}(\nu_N,\pmeas)
 &\leq \int_{\Xsurf\times \Xsurf}d(x,y)^\alpha\,\diff\pi(x,y)\\
 &\leq\left(\int_{\Xsurf\times \Xsurf}d(x,y)\,\diff\pi(x,y)\right)^\alpha
 =W_1(\nu_N,\pmeas)^\alpha.
 \end{aligned}
\]
The upper bound follows from \eqref{eq:Wasserstein-bound}; if its constant
is $C_0$, then $C_0^\alpha\leq\max\{1,C_0\}$, so the constant is uniform in
$\alpha$.  The general lower-bound assertion of
\Cref{thm:energy-concave-modulus}, applied to the concave modulus
$\omega(t)=t^\alpha$, holds for every probability measure supported on at most
$N$ points and supplies the other inequality, also uniformly in $\alpha$.
\end{proof}

\begin{theorem}[Sharp discrepancy for concave moduli]
\label{thm:energy-concave-modulus}
Let $\omega:[0,\infty)\to[0,\infty)$ be continuous, nondecreasing, and
concave, with $\omega(0)=0$ and $\omega(t)>0$ for $t>0$.  For such $\omega$, write
\[
 \mathcal D_{\omega}(\nu,\pmeas)
 :=\sup_{\substack{f\in C(\Xsurf;\mathbb R)\\
          |f(x)-f(y)|\leq\omega(d_{\geo}(x,y))\ \text{for all }x,y}}
       \left|\int_{\Xsurf}f\,\diff(\nu-\pmeas)\right|.
\]
Under the hypotheses of \Cref{cor:Wasserstein},
\begin{equation}\label{eq:energy-concave-modulus}
 c_{\Xsurf}\omega(N^{-1/2})
 \leq\mathcal D_{\omega}(\nu_N,\pmeas)
 \leq C_{\Xsurf,K,\beta}\omega(N^{-1/2}).
\end{equation}
The constants may be chosen independently of $\omega$.  More generally,
the lower bound holds for every probability measure supported on at most
$N$ points.
\end{theorem}

\begin{proof}
Concavity and $\omega(0)=0$ imply that $t\mapsto\omega(t)/t$ is
nonincreasing on $(0,\infty)$.  The subadditivity cases with $a=0$ or $b=0$
are immediate.  If $a\geq b>0$, the monotonicity of this quotient gives
\[
 \omega(a+b)
 \leq\left(1+\frac ba\right)\omega(a)
 \leq\omega(a)+\omega(b).
\]
Thus $\omega$ is subadditive, and consequently
\begin{equation}\label{eq:concave-modulus-difference}
 |\omega(a)-\omega(b)|\leq\omega(|a-b|)
 \qquad(a,b\geq0).
\end{equation}
It follows that
$d_{\omega}(x,y):=\omega(d_{\geo}(x,y))$ is a metric.  By
Kantorovich--Rubinstein duality, $\mathcal D_{\omega}$ is exactly the
$W_1$ transport cost for $d_{\omega}$:
\[
 \mathcal D_{\omega}(\nu,\pmeas)
 =\inf_{\pi\in\Pi(\nu,\pmeas)}
   \int_{\Xsurf\times \Xsurf}d_{\omega}(x,y)\,\diff\pi(x,y).
\]

Let $\pi$ be an optimal coupling for the original distance $d_{\geo}$.
Jensen's inequality for the concave function $\omega$ gives
\[
 \begin{aligned}
 \mathcal D_{\omega}(\nu_N,\pmeas)
 &\leq\int_{\Xsurf\times \Xsurf}\omega(d_{\geo}(x,y))\,\diff\pi(x,y)\\
 &\leq\omega\!\left(\int_{\Xsurf\times \Xsurf}d_{\geo}(x,y)\,\diff\pi(x,y)\right)
 =\omega\bigl(W_1(\nu_N,\pmeas)\bigr).
 \end{aligned}
\]
If $C_0$ is the constant in \eqref{eq:Wasserstein-bound}, monotonicity,
concavity, and $\omega(0)=0$ imply
\[
 \omega(C_0t)\leq\max\{1,C_0\}\omega(t).
\]
This proves the upper bound in \eqref{eq:energy-concave-modulus} with a
constant independent of $\omega$.

For the lower bound, let $S=\operatorname{supp}\nu$ with $\#S\leq N$.
Choose $C_{\rm vol}<\infty$ and $r_*>0$ such that
$\pmeas(B_{\geo}(x,r))\leq C_{\rm vol}r^2$ for $x\in \Xsurf$ and $0<r\leq r_*$. 
After enlarging $C_{\rm vol}$, assume
$(2C_{\rm vol})^{-1/2}\leq r_*$, and write
\[
 r_N=(2C_{\rm vol}N)^{-1/2}=c_0N^{-1/2}.
\]
The complement of the union of the $r_N$-balls around $S$ has measure at
least $1/2$.  By \eqref{eq:concave-modulus-difference} and the
one-Lipschitz property of $d_{\geo}(\,\cdot\,,S)$, the function
\[
 \psi(x)=\omega(d_{\geo}(x,S))
\]
is one-Lipschitz for $d_{\omega}$ and vanishes on $S$.  Hence
\[
 \mathcal D_{\omega}(\nu,\pmeas)
 \geq\int_{\Xsurf}\psi\,\diff\pmeas
 \geq\frac12\omega(c_0N^{-1/2}).
\]
Finally,
$\omega(c_0t)\geq\min\{1,c_0\}\omega(t)$: for $c_0\leq1$ this is
concavity, and for $c_0\geq1$ it is monotonicity.  This proves the lower
bound uniformly in $\omega$.
\end{proof}

\begin{proposition}[Metric-dimensional modulus quantization]
\label{prop:abstract-modulus-quantization}
Let $(Y,d,\mathfrak m)$ be a compact metric probability space, and suppose
that for some $m>0$, $C_{\rm vol}<\infty$, and $r_0>0$,
\begin{equation}\label{eq:abstract-small-ball-upper}
 \mathfrak m(B_d(y,r))\leq C_{\rm vol}r^m
 \qquad(y\in Y,\ 0<r\leq r_0).
\end{equation}
For $\omega$ satisfying the hypotheses of
\Cref{thm:energy-concave-modulus}, write
\[
 \mathcal D_{\omega,d}(\nu,\mathfrak m)
 :=\sup_{\substack{f\in C(Y;\mathbb R)\\
          |f(x)-f(y)|\leq\omega(d(x,y))\ \text{for all }x,y}}
       \left|\int_Yf\,\diff(\nu-\mathfrak m)\right|.
\]
Then every probability measure $\nu$ on $Y$ satisfies
\begin{equation}\label{eq:abstract-modulus-upper}
 \mathcal D_{\omega,d}(\nu,\mathfrak m)
 \leq\omega\bigl(W_1^d(\nu,\mathfrak m)\bigr).
\end{equation}
If $\#\operatorname{supp}\nu\leq N$ and
$r_N=(2C_{\rm vol}N)^{-1/m}\leq r_0$, then
\begin{equation}\label{eq:abstract-modulus-lower}
 \mathcal D_{\omega,d}(\nu,\mathfrak m)
 \geq\frac12\omega(r_N).
\end{equation}
Consequently, if
$W_1^d(\nu_N,\mathfrak m)\leq A N^{-1/m}$ and
$\#\operatorname{supp}\nu_N\leq N$, then, for all sufficiently large $N$,
\begin{equation}\label{eq:abstract-modulus-two-sided}
 c_{m,C_{\rm vol}}\omega(N^{-1/m})
 \leq\mathcal D_{\omega,d}(\nu_N,\mathfrak m)
 \leq C_A\omega(N^{-1/m}),
\end{equation}
with constants independent of $\omega$.
The threshold for ``sufficiently large'' may depend on $r_0$.
\end{proposition}

\begin{proof}
Concavity makes $\omega\circ d$ a metric, so the modulus discrepancy is itself a
Kantorovich--Rubinstein distance; Jensen then transfers the transport bound
through $\omega$, and the lower bound is a direct test against
$\omega(d(\cdot,S))$.

By \eqref{eq:concave-modulus-difference},
$d_{\omega}=\omega\circ d$ is a metric.  If $\pi$ is an optimal coupling for
$d$, Kantorovich--Rubinstein duality and Jensen's inequality give
\[
 \mathcal D_{\omega,d}(\nu,\mathfrak m)
 \leq\int_{Y\times Y}\omega(d(x,y))\,\diff\pi
 \leq\omega\!\left(\int_{Y\times Y}d(x,y)\,\diff\pi\right),
\]
which is \eqref{eq:abstract-modulus-upper}.

For the lower bound, let $S=\operatorname{supp}\nu$.  By
\eqref{eq:abstract-small-ball-upper}, the union of the $r_N$-balls centered
at $S$ has $\mathfrak m$-measure at most $1/2$.  The function
$y\mapsto\omega(d(y,S))$ is one-Lipschitz for $d_{\omega}$, vanishes on $S$,
and is at least $\omega(r_N)$ outside that union.  Testing with this function
proves \eqref{eq:abstract-modulus-lower}.

Finally, write $r_N=c_0N^{-1/m}$ with
$c_0=(2C_{\rm vol})^{-1/m}$.  Concavity and monotonicity give
\[
 \omega(c_0t)\geq\min\{1,c_0\}\omega(t),
 \qquad
 \omega(At)\leq\max\{1,A\}\omega(t).
\]
Combining these inequalities with
\eqref{eq:abstract-modulus-upper} and
\eqref{eq:abstract-modulus-lower} proves
\eqref{eq:abstract-modulus-two-sided} uniformly in $\omega$.
\end{proof}

\section{Mesoscopic counting, multiplicity, and negative Sobolev norms}
\label{sec:mesoscopic}

\subsection{Mesoscopic estimates and multiplicities}

The discrepancy estimate of \Cref{prop:energy-to-discrepancy} may be tested
against scaled bumps and against the heat kernel.  This gives two-sided point
counts down to the scale $N^{-1/4}$, uniform heat regularization, the
multiplicity bound which replaces separation under the energy hypothesis, and
negative Sobolev convergence.

\begin{proposition}[Two-sided mesoscopic counting from Green energy]
\label{prop:energy-mesoscopic-count}
Under the hypotheses of \Cref{cor:Wasserstein}, there are $r_*>0$ and
$C_{\Xsurf,K,\beta}<\infty$ such that
\begin{align}
 N\pmeas(B_{\geo}(x,r/2))-C_{\Xsurf,K,\beta}r^{-2}
 &\leq \#\bigl\{j:w_{j,N}\in B_{\geo}(x,r)\bigr\}
 \label{eq:energy-mesoscopic-count-lower}\\
 &\leq N\pmeas(B_{\geo}(x,2r))+C_{\Xsurf,K,\beta}r^{-2}
 \label{eq:energy-mesoscopic-count-upper}
\end{align}
for every $x\in \Xsurf$, every $0<r\leq r_*$, and every $N$.  Repeated entries
are counted with multiplicity.  In particular,
\begin{equation}\label{eq:energy-mesoscopic-count-volume}
 c_{\Xsurf}Nr^2-C_{\Xsurf,K,\beta}r^{-2}
 \leq\#\bigl\{j:w_{j,N}\in B_{\geo}(x,r)\bigr\}
 \leq C_{\Xsurf}Nr^2+C_{\Xsurf,K,\beta}r^{-2}.
\end{equation}
\end{proposition}

\begin{proof}
After decreasing $r_*$ once, choose standard inner and outer scaled cutoffs
$\phi^-_{x,r}$ and $\phi^+_{x,r}$ such that
\[
 0\leq\phi^-_{x,r}\leq\mathbf1_{B_{\geo}(x,r)},
 \qquad
 \phi^-_{x,r}=1\ \text{on }B_{\geo}(x,r/2),
 \qquad
 \mathbf1_{B_{\geo}(x,r)}\leq\phi^+_{x,r}\leq
 \mathbf1_{B_{\geo}(x,2r)},
\]
and
\[
 \|\ddc\phi^-_{x,r}\|_{\infty}
 +\|\ddc\phi^+_{x,r}\|_{\infty}\leq C_{\Xsurf}r^{-2}.
\]
Such cutoffs exist uniformly for $0<r\leq r_*$ by the scaled bump
construction already used in \eqref{eq:covering-bump-estimates}.  The inner
cutoff gives
\[
 \begin{aligned}
 \nu_N(B_{\geo}(x,r))
 &\geq\int_{\Xsurf}\phi^-_{x,r}\,\diff\nu_N\\
 &\geq\int_{\Xsurf}\phi^-_{x,r}\,\diff\pmeas
   -\left|\int_{\Xsurf}\phi^-_{x,r}\,\diff(\nu_N-\pmeas)\right|\\
 &\geq\pmeas(B_{\geo}(x,r/2))-\frac{C_{\Xsurf,K,\beta}}{Nr^2},
 \end{aligned}
\]
where the last line uses \eqref{eq:energy-discrepancy}.  Similarly, the
outer cutoff gives
\[
 \begin{aligned}
 \nu_N(B_{\geo}(x,r))
 &\leq\int_{\Xsurf}\phi^+_{x,r}\,\diff\nu_N\\
 &\leq\int_{\Xsurf}\phi^+_{x,r}\,\diff\pmeas
   +\left|\int_{\Xsurf}\phi^+_{x,r}\,\diff(\nu_N-\pmeas)\right|\\
 &\leq\pmeas(B_{\geo}(x,2r))+\frac{C_{\Xsurf,K,\beta}}{Nr^2}.
 \end{aligned}
\]
Multiplying by $N$ proves
\eqref{eq:energy-mesoscopic-count-lower} and
\eqref{eq:energy-mesoscopic-count-upper}.
Uniform small-ball volume comparison gives
\eqref{eq:energy-mesoscopic-count-volume}.
\end{proof}

\begin{corollary}[Uniform density above the discrepancy scale]
\label{cor:energy-mesoscopic-density}
There is $A_{\Xsurf,K,\beta}<\infty$ such that, whenever
\[
 A_{\Xsurf,K,\beta}N^{-1/4}\leq r\leq r_*,
\]
one has
\begin{equation}\label{eq:energy-mesoscopic-density}
 c_{\Xsurf,K,\beta}Nr^2
 \leq\#\bigl\{j:w_{j,N}\in B_{\geo}(x,r)\bigr\}
 \leq C_{\Xsurf,K,\beta}Nr^2
\end{equation}
for every $x\in \Xsurf$.
\end{corollary}

\begin{proof}
Choose $A_{\Xsurf,K,\beta}$ so large that, whenever
\[
 r\geq A_{\Xsurf,K,\beta}N^{-1/4},
\]
the error term
$C_{\Xsurf,K,\beta}r^{-2}$ in
\eqref{eq:energy-mesoscopic-count-volume} is at most half of its lower main
term $c_{\Xsurf}Nr^2$.  The lower bound follows, and the same comparison absorbs
the upper error term into $C_{\Xsurf}Nr^2$.
\end{proof}

\begin{corollary}[Mesoscopic heat regularization]
\label{cor:energy-heat-regularization}
Let $H_t(x,y)$ be the kernel, with respect to $\pmeas$, of
$e^{-t\Delta_{\geo_{\pmeas}}}$ for a fixed conformal metric $\geo_{\pmeas}$ whose volume
form is $\pmeas$, and let the associated smooth density be
\[
 \rho_{N,t}(x):=\int_{\Xsurf}H_t(x,y)\,\diff\nu_N(y).
\]
Under the hypotheses of \Cref{cor:Wasserstein}, there are $t_0>0$ and
$C_{\Xsurf,K,\beta}<\infty$ such that
\begin{equation}\label{eq:energy-heat-uniform-error}
 \|\rho_{N,t}-1\|_{L^\infty(\Xsurf)}
 \leq\frac{C_{\Xsurf,K,\beta}}{Nt^2}
 \qquad(0<t\leq t_0).
\end{equation}
Consequently, for some $A^{\rm heat}_{\Xsurf,K,\beta}<\infty$ and every
sufficiently large $N$, the interval below is nonempty and
\begin{equation}\label{eq:energy-heat-positive-density}
 \frac12\leq\rho_{N,t}(x)\leq\frac32
\end{equation}
for every $x\in \Xsurf$ whenever
$A^{\rm heat}_{\Xsurf,K,\beta}N^{-1/2}\leq t\leq t_0$.
\end{corollary}

\begin{proof}
The heat kernel preserves constants, so
$\int_{\Xsurf}H_t(x,y)\,\diff\pmeas(y)=1$.  For fixed $x$ and $t>0$, apply
\eqref{eq:energy-discrepancy} to the smooth test function
$y\mapsto H_t(x,y)$.  Standard derivative bounds for the heat kernel on a
compact two-manifold give
\begin{equation}\label{eq:heat-kernel-second-derivative}
 \sup_{x\in \Xsurf}\|\ddc_yH_t(x,\cdot)\|_{\infty}\leq C_{\Xsurf}t^{-2}
 \qquad(0<t\leq t_0).
\end{equation}
Indeed, the Gaussian bound for two derivatives is
$C_{\Xsurf}t^{-(2+2)/2}\exp(-d_{\geo_{\pmeas}}(x,y)^2/(C_{\Xsurf}t))$; the fixed metric compares
$\ddc_yH_t$ with the $y$-Hessian.  Hence
\[
 \begin{aligned}
 |\rho_{N,t}(x)-1|
 &=\left|\int_{\Xsurf}H_t(x,y)\,\diff(\nu_N-\pmeas)(y)\right|\\
 &\leq\frac{C_{\Xsurf,K,\beta}}{N}
        \|\ddc_yH_t(x,\cdot)\|_{\infty}
 \leq\frac{C_{\Xsurf,K,\beta}}{Nt^2}.
 \end{aligned}
\]
Taking the supremum in $x$ proves
\eqref{eq:energy-heat-uniform-error}.  Choose
$A^{\rm heat}_{\Xsurf,K,\beta}$ so that
$C_{\Xsurf,K,\beta}(A^{\rm heat}_{\Xsurf,K,\beta})^{-2}\leq1/2$.
Then $t\geq A^{\rm heat}_{\Xsurf,K,\beta}N^{-1/2}$ makes the right-hand side at most
$1/2$, proving \eqref{eq:energy-heat-positive-density}.
\end{proof}

\begin{corollary}[Square-root multiplicity from energy alone]
\label{cor:energy-multiplicity}
Under the hypotheses of \Cref{cor:Wasserstein},
\begin{equation}\label{eq:energy-atom-bound}
 \sup_{x\in \Xsurf}\nu_N(\{x\})
 \leq C_{\Xsurf,K,\beta}N^{-1/2}.
\end{equation}
Thus, if the entries of $\Wset_N$ are counted with multiplicity, then
\begin{equation}\label{eq:energy-multiplicity-bound}
 \max_{x\in \Xsurf}\#\{j:w_{j,N}=x\}
 \leq C_{\Xsurf,K,\beta}\sqrt N.
\end{equation}
\end{corollary}

\begin{proof}
For all sufficiently large $N$, apply
\eqref{eq:energy-mesoscopic-count-upper} with $r=N^{-1/4}\leq r_*$.  Since
an atom at $x$ lies in every centered ball, division by $N$ and small-ball
volume comparison give
\[
 \nu_N(\{x\})
 \leq C_{\Xsurf,K,\beta}\bigl(r^2+N^{-1}r^{-2}\bigr)
 \leq C_{\Xsurf,K,\beta}N^{-1/2}.
\]
Enlarging the constant handles the finitely many smaller $N$ and proves
\eqref{eq:energy-atom-bound}.  Multiplication by $N$ gives
\eqref{eq:energy-multiplicity-bound}.
\end{proof}

\begin{theorem}[Negative Sobolev discrepancy from Green energy]
\label{thm:energy-negative-Sobolev}
Fix a nonnegative Laplace--Beltrami operator on $\Xsurf$ and let it determine the
Hilbert Sobolev scale $H^s(\Xsurf)$.  Under the hypotheses of
\Cref{cor:Wasserstein}, write $\sigma_N=\nu_N-\pmeas$.  Then, for every
$1<s\leq3$ and every $\varepsilon>0$,
\begin{equation}\label{eq:energy-negative-Sobolev-intermediate}
 \|\sigma_N\|_{H^{-s}}
 \leq C_{\Xsurf,K,\beta,s,\varepsilon}
       N^{-(s-1)/2+\varepsilon}.
\end{equation}
For every $s>3$,
\begin{equation}\label{eq:energy-negative-Sobolev-C2}
 \|\sigma_N\|_{H^{-s}}
 \leq C_{\Xsurf,K,\beta,s}N^{-1}.
\end{equation}
Changing the fixed Laplace--Beltrami operator replaces these norms by
equivalent ones and changes only the constants.
\end{theorem}

\begin{proof}
We use three estimates.  First, because $s>1$ implies
$H^s(\Xsurf)\hookrightarrow C^0(\Xsurf)$ in real dimension two, every difference of
probability measures has uniformly bounded $H^{-s}$ norm.  This disposes of
values of $\varepsilon$ for which the right-hand side of
\eqref{eq:energy-negative-Sobolev-intermediate} does not decay.  Thus, in
the first range below, it suffices to consider
$0<\varepsilon<(s-1)/2$.

Second, take $1<s\leq2$ and
$0<\varepsilon<(s-1)/2$, and write
$\alpha=s-1-2\varepsilon\in(0,1)$.  The strict Sobolev embedding
\[
 H^s(\Xsurf)\hookrightarrow C^{0,\alpha}(\Xsurf)
\]
and \eqref{eq:energy-holder-discrepancy} give, for real smooth $f$,
\[
 |\langle\sigma_N,f\rangle|
 \leq C_{s,\varepsilon}N^{-\alpha/2}\|f\|_{H^s}
 =C_{s,\varepsilon}N^{-(s-1)/2+\varepsilon}\|f\|_{H^s}.
\]
Density of smooth functions proves
\eqref{eq:energy-negative-Sobolev-intermediate} in this range.

Third, fix any $s_0>2$.  Since $\sigma_N(\Xsurf)=0$, constants in test functions
play no role.  The embedding
$H^{s_0}(\Xsurf)\hookrightarrow C^1(\Xsurf)$ and Kantorovich--Rubinstein duality give
\begin{equation}\label{eq:negative-Sobolev-W1-endpoint}
 \|\sigma_N\|_{H^{-s_0}}
 \leq C_{s_0}W_1(\nu_N,\pmeas)
 \leq C_{\Xsurf,K,\beta,s_0}N^{-1/2}.
\end{equation}
If $s_1>3$, the embedding $H^{s_1}(\Xsurf)\hookrightarrow C^2(\Xsurf)$ and
\eqref{eq:energy-discrepancy} instead give
\begin{equation}\label{eq:negative-Sobolev-C2-endpoint}
 \|\sigma_N\|_{H^{-s_1}}
 \leq C_{\Xsurf,K,\beta,s_1}N^{-1}.
\end{equation}
This proves \eqref{eq:energy-negative-Sobolev-C2} directly.

It remains to interpolate when $2<s\leq3$.  Choose
\[
 0<\delta<\min\{s-2,2\varepsilon\},
 \qquad s_0=2+\delta,
 \qquad s_1=3+\delta,
\]
and set $\theta=s-s_0\in(0,1)$.  Since $s_1-s_0=1$, the Hilbert-scale
interpolation inequality, together with
\eqref{eq:negative-Sobolev-W1-endpoint} and
\eqref{eq:negative-Sobolev-C2-endpoint}, yields
\[
 \begin{aligned}
 \|\sigma_N\|_{H^{-s}}
 &\leq C
   \|\sigma_N\|_{H^{-s_0}}^{1-\theta}
   \|\sigma_N\|_{H^{-s_1}}^\theta\\
 &\leq C N^{-\frac12(1-\theta)-\theta}
 =C N^{-(s-1-\delta)/2}
 \leq C N^{-(s-1)/2+\varepsilon}.
 \end{aligned}
\]
This includes $s=3$ and completes the proof.
\end{proof}

\section{The optimal constants as functions of the exponent}
\label{sec:optimal-constants}

We now turn from configurations satisfying \eqref{eq:energy-hypothesis} to the
optimal constants themselves.  The exponent calculus and the variational
arguments are for fixed $N$; uniformity in $N$ is stated explicitly where
proved, in particular in \Cref{thm:theta-uniform-noncollapse}.

\subsection{Optimal constants and microscopic limits}

For $\Wset\in \Xsurf^N$, write
\[
 V_{\Wset}=\sum_{w\in \Wset}U_{\Xsurf}(\,\cdot\,,w),
 \qquad Y_{\Wset}=e^{V_{\Wset}},
\]
and the pointwise endpoint
\begin{equation}\label{eq:theta-infinity}
 \Theta_{N,\infty}(\Xsurf)
 :=\inf_{\Wset\in \Xsurf^N}\|Y_{\Wset}\|_{L^\infty(\Xsurf)}
 =\exp\!\left(\inf_{\Wset\in \Xsurf^N}\sup_{\Xsurf}V_{\Wset}\right).
\end{equation}
The factor $e^{U_{\Xsurf}(z,w)}$ extends continuously by zero across $z=w$, so
$Y_{\Wset}$ is a continuous nonnegative function of $(z,\Wset)\in \Xsurf\times \Xsurf^N$.
Thus the infimum in \eqref{eq:theta-infinity} is attained.  The geometric
mean of $Y_{\Wset}$ is one, and therefore
\begin{equation}\label{eq:theta-infinity-bounds}
 1<\Theta_{N,\infty}(\Xsurf)\leq e^{C_{\Xsurf,\pmeas}}.
\end{equation}
The upper bound is \Cref{thm:packing-input}.  For strictness, an attaining
function
with $\|Y_{\Wset}\|_{\infty}\leq1$ would have
$V_{\Wset}=\log Y_{\Wset}\leq0$ and mean zero, hence $V_{\Wset}=0$ almost everywhere.  This
is impossible because $Y_{\Wset}$ tends to zero at every node and is continuous.

\begin{problem}[Optimal constants]\label{prob:optimal-constants}
Determine the asymptotic behavior, or the limit if it exists, of
$\Theta_{N,\beta}(\Xsurf)$ and $\Theta_{N,\infty}(\Xsurf)$ as $N\to\infty$.
\end{problem}

Two things can be said in complete generality before any microscopic analysis.
First, neither optimal constant can collapse to the Jensen value one: the
following theorem shows that the excess is bounded below uniformly in $N$, at the
exact order $\beta^2$ for small $\beta$.  Second, on the round sphere and for
$N=2$ the whole problem can be solved in closed form; see
\Cref{subsec:sphere-two-point-model}.

\begin{theorem}[Uniform non-collapse of the optimal constants]
\label{thm:theta-uniform-noncollapse}
There are constants $a_{\Xsurf},c_{\Xsurf},C_{\Xsurf}>0$, independent of $N$, such that the
following statements hold.
\begin{enumerate}[label=\textup{(\roman*)}]
\item Every $\Wset\in \Xsurf^N$ satisfies
\begin{equation}\label{eq:theta-uniform-L1-gap}
 \|V_{\Wset}\|_{L^1(\Xsurf,\pmeas)}\ge a_{\Xsurf}.
\end{equation}
\item For every $\beta>0$,
\begin{equation}\label{eq:theta-uniform-finite-gap}
 \Theta_{N,\beta}(\Xsurf)\ge1+c_{\Xsurf}\beta^2.
\end{equation}
\item The pointwise endpoint satisfies
\begin{equation}\label{eq:theta-uniform-infinity-gap}
 \Theta_{N,\infty}(\Xsurf)\ge e^{a_{\Xsurf}/2}>1.
\end{equation}
\item For $0<\beta\le1$,
\begin{equation}\label{eq:theta-uniform-small-beta}
 1+c_{\Xsurf}\beta^2
 \le\Theta_{N,\beta}(\Xsurf)
 \le1+C_{\Xsurf}\beta^2.
\end{equation}
\end{enumerate}
Consequently neither $\Theta_{N,\beta}(\Xsurf)$, for fixed $\beta>0$, nor
$\Theta_{N,\infty}(\Xsurf)$ can converge to the Jensen value one along a
subsequence.
\end{theorem}

\begin{proof}
The argument has two halves.  First we show that $A_{\Wset}=\|V_{\Wset}\|_{L^1}$ cannot be
small, by playing the heat-smoothed upper bound for $W_1$ against the atomic
lower bound; the optimal choice of the smoothing time is what produces the
constant.  Then we convert that into a lower bound for the exponential energy,
using the mean-zero identity and convexity of $e^x-1-x$.

Fix a Riemannian distance on $\Xsurf$ and write
$A_{\Wset}:=\|V_{\Wset}\|_{L^1(\Xsurf,\pmeas)}$.  This quantity is positive: otherwise
$V_{\Wset}=0$ as a distribution, contradicting $\ddc V_{\Wset}=N(\nu_{\Wset}-\pmeas)$ because
$\nu_{\Wset}$ is atomic and $\pmeas$ is smooth. The heat-smoothing estimates used in the proof
of \Cref{cor:Wasserstein} give, for every $0<t\le1$,
\begin{equation}\label{eq:theta-potential-transport}
 W_1(\nu_{\Wset},\pmeas)
 \le C_{\Xsurf}\left(\sqrt t+\frac{A_{\Wset}}{N\sqrt t}\right).
\end{equation}
Indeed, for a one-Lipschitz test function $\phi$ and
$\phi_t=e^{-t\Delta_{\pmeas}}\phi$,
\[
 \|\phi-\phi_t\|_{\infty}\le C_{\Xsurf}\sqrt t,
 \qquad
 \|\ddc\phi_t\|_{\infty}\le C_{\Xsurf}t^{-1/2},
\]
and the identity $\ddc V_{\Wset}=N(\nu_{\Wset}-\pmeas)$ gives
\[
 \left|\int_{\Xsurf}\phi_t\,\diff(\nu_{\Wset}-\pmeas)\right|
 \le\frac{C_{\Xsurf}A_{\Wset}}{N\sqrt t}.
\]
On the other hand, \Cref{cor:Wasserstein-lower} gives
$W_1(\nu_{\Wset},\pmeas)\ge c_{\Xsurf}N^{-1/2}$ because $\nu_{\Wset}$ is supported on at most
$N$ points.  If $A_{\Wset}\le N$, choose $t=A_{\Wset}/N$ in
\eqref{eq:theta-potential-transport}; if $A_{\Wset}>N$, there is nothing to prove.
After decreasing the constant, this yields
\eqref{eq:theta-uniform-L1-gap}.

Since $\int_{\Xsurf}V_{\Wset}\,\diff\pmeas=0$,
\[
 \int_{\Xsurf}(V_{\Wset})_+\,\diff\pmeas=\frac12\|V_{\Wset}\|_1.
\]
Moreover, $e^x-1-x\ge x^2/2$ for $x\ge0$.  Therefore
\begin{align*}
 \calB_{\beta}(\Wset;\Xsurf)-1
 &=\int_{\Xsurf}\bigl(e^{\beta V_{\Wset}}-1-\beta V_{\Wset}\bigr)\,\diff\pmeas\\
 &\ge\frac{\beta^2}{2}\int_{\{V_{\Wset}>0\}}V_{\Wset}^2\,\diff\pmeas
 \ge\frac{\beta^2}{8}\|V_{\Wset}\|_1^2.
\end{align*}
Taking the infimum proves \eqref{eq:theta-uniform-finite-gap}.  Likewise,
\[
 \sup_{\Xsurf}V_{\Wset}\ge\int_{\Xsurf}(V_{\Wset})_+\,\diff\pmeas
 =\frac12\|V_{\Wset}\|_1,
\]
which proves \eqref{eq:theta-uniform-infinity-gap}.

It remains to prove the upper estimate in
\eqref{eq:theta-uniform-small-beta}.  Choose specifically the holomorphic
divisors $\Wset_N$ of \cite[Proposition~6.1.2 and Theorem~7.1.1]{PaperA}.  For the
corresponding sections, let $F_N=\log|s_N|_{\heo^N}$ and
$m_N=\int_{\Xsurf}F_N\,\diff\pmeas$.  The divisor identity gives $V_{\Wset_N}=F_N-m_N$,
with $|m_N|\le C_{\Xsurf}$ and $\sup_{\Xsurf}V_{\Wset_N}\le C_{\Xsurf}$.  Outside the disjoint
rescaled zero balls, $|F_N|\le C_{\Xsurf}$.  On each ball, the local norm model
gives $|F_N|\le C_{\Xsurf}+|\log|u||$ and $\diff\pmeas\le C_{\Xsurf}N^{-1}dA(u)$.
The coordinate images of these balls lie in a fixed disk $|u|<R$.  Hence
\[
 \int_{\Xsurf}F_N^2\,\diff\pmeas
 \le C_{\Xsurf}+\sum_{j=1}^N\frac{C_{\Xsurf}}{N}
       \int_{|u|<R}(1+|\log|u||)^2\,dA(u)\le C_{\Xsurf},
\]
so $\|V_{\Wset_N}\|_{L^2(\Xsurf,\pmeas)}\le C_{\Xsurf}$.  Fixed configurations of distinct
points handle the finitely many degrees below the construction threshold,
after increasing $C_{\Xsurf}$.
For $0<\beta\le1$, the elementary Taylor remainder estimate, separately on
$\{V_{\Wset_N}\le0\}$ and $\{V_{\Wset_N}>0\}$, gives
\[
 0\le e^{\beta V_{\Wset_N}}-1-\beta V_{\Wset_N}
 \le C_{\Xsurf}\beta^2V_{\Wset_N}^2.
\]
Integration and the mean-zero identity imply
$\calB_{\beta}(\Wset_N;\Xsurf)\le1+C_{\Xsurf}\beta^2$.  The lower estimate is already
\eqref{eq:theta-uniform-finite-gap}.
\end{proof}

\begin{remark}[What remains of the optimal-constants problem]
\Cref{thm:theta-uniform-noncollapse} settles the previously open possibility
that the optimal constants might collapse to one.  The remaining problem is
to prove existence of the large-$N$ limits, identify their values, and
classify the associated microscopic minimizers.
\end{remark}

The exponent enters $\Theta_{N,\beta}$ in two ways at once, through the integrand
and through the normalization, so it is worth recording the resulting monotonicity
before anything finer.  The non-strict monotonicity on the sphere was noted in
\cite[Remark~2.1.2]{EtayoHedenmalmOrtega}.

\begin{proposition}[Strict monotonicity in the exponent]
\label{prop:theta-beta-monotone}
For fixed $N$ and $\Xsurf$, the function
\[
 \beta\longmapsto\Theta_{N,\beta}(\Xsurf)^{1/\beta}
\]
is strictly increasing on $(0,\infty)$.  More precisely, if
$0<\beta_1<\beta_2<\infty$, then
\[
 1<\Theta_{N,\beta_1}(\Xsurf)^{1/\beta_1}
 <\Theta_{N,\beta_2}(\Xsurf)^{1/\beta_2}
 <\Theta_{N,\infty}(\Xsurf).
\]
\end{proposition}

\begin{proof}
For a fixed configuration, write $Y_{\Wset}(z)=e^{\sum_jU_{\Xsurf}(z,w_j)}$.  On the
probability space $(\Xsurf,\pmeas)$, monotonicity of power means gives
\[
 \calB_{\beta_1}(\Wset;\Xsurf)^{1/\beta_1}
 =\|Y_{\Wset}\|_{L^{\beta_1}}
 \leq\|Y_{\Wset}\|_{L^{\beta_2}}
 =\calB_{\beta_2}(\Wset;\Xsurf)^{1/\beta_2}
\]
whenever $0<\beta_1<\beta_2$.  The inequality is strict for every $\Wset$:
$Y_{\Wset}$ is continuous, vanishes at each node, and is positive away from the
finitely many nodes, so it is not constant almost everywhere.  Joint
continuity of $Y_{\Wset}(z)$ on the compact space $\Xsurf\times \Xsurf^N$ implies that
$F_{\beta_2}(\Wset):=\|Y_{\Wset}\|_{L^{\beta_2}}$ attains its minimum.  If
$\Wset_{\beta_2}$ is a minimizer, then
\[
 \inf_{\Wset}\|Y_{\Wset}\|_{L^{\beta_1}}
 \leq\|Y_{\Wset_{\beta_2}}\|_{L^{\beta_1}}
 <\|Y_{\Wset_{\beta_2}}\|_{L^{\beta_2}}
 =\inf_{\Wset}\|Y_{\Wset}\|_{L^{\beta_2}}.
\]
The same argument with a minimizer of the continuous function
$\Wset\mapsto\|Y_{\Wset}\|_{\infty}$ proves the last strict inequality.  Finally,
strict Jensen gives $\|Y_{\Wset}\|_{L^{\beta_1}}>1$ for every $\Wset$, since $V_{\Wset}$
has mean zero and is not almost everywhere constant.  Compact attainment
then gives the first strict inequality.
\end{proof}

\begin{proposition}[The exponent endpoints]
\label{prop:theta-beta-endpoints}
For every fixed positive integer $N$,
\begin{equation}\label{eq:theta-beta-endpoints}
 \lim_{\beta\downarrow0}\Theta_{N,\beta}(\Xsurf)^{1/\beta}=1,
 \qquad
 \lim_{\beta\to\infty}\Theta_{N,\beta}(\Xsurf)^{1/\beta}
 =\Theta_{N,\infty}(\Xsurf).
\end{equation}
Moreover, $\Theta_{N,\beta}(\Xsurf)>1$ for every fixed $N$ and $\beta>0$.
\end{proposition}

\begin{proof}
Since $\int_{\Xsurf}V_{\Wset}\,\diff\pmeas=0$, Jensen gives
$\|Y_{\Wset}\|_{L^\beta}\geq1$.  For one fixed configuration $\Wset$, the function
$V_{\Wset}$ is bounded above and has only logarithmic negative singularities, so
$V_{\Wset}\in L^1$.  Dominated convergence applied to the difference quotient at
$\beta=0$ gives
\[
 \log\|Y_{\Wset}\|_{L^\beta}
 =\frac1\beta\log\int_{\Xsurf}e^{\beta V_{\Wset}}\,\diff\pmeas\longrightarrow0.
\]
Taking the infimum over configurations and combining this upper comparison
with Jensen proves the first limit.

For the second, consider on the compact space $\Xsurf^N$ the functions
\[
 F_p(\Wset):=\|Y_{\Wset}\|_{L^p(\Xsurf)},\qquad
 F_{\infty}(\Wset):=\|Y_{\Wset}\|_{L^\infty(\Xsurf)}.
\]
Joint continuity of $Y_{\Wset}(z)$ on $\Xsurf\times \Xsurf^N$ makes all these functions
continuous.  For positive integers $p$, the sequence $F_p$ is increasing
and converges pointwise to the continuous function $F_{\infty}$.  Dini's
theorem makes the convergence uniform on $\Xsurf^N$.  Hence
\[
 \lim_{p\to\infty}\inf_{\Wset\in \Xsurf^N}F_p(\Wset)
 =\inf_{\Wset\in \Xsurf^N}F_{\infty}(\Wset)=\Theta_{N,\infty}(\Xsurf).
\]
Monotonicity of power means squeezes arbitrary real $\beta\to\infty$
between adjacent integers and proves the second limit.

Finally, $\calB_{\beta}(\,\cdot\,;\Xsurf)$ attains its minimum on the compact space $\Xsurf^N$.
Indeed $e^{\gamma U_{\Xsurf}}$ extends continuously by zero across the diagonal for every
$\gamma>0$ by \Cref{lem:green-joint-local}, so $Y_{\Wset}(z)^\beta$ is jointly continuous on
the compact space $\Xsurf\times \Xsurf^N$, whence
$\calB_{\beta}(\,\cdot\,;\Xsurf)$ is continuous; this is the argument written out in full in
\Cref{prop:minimizer-exists}, which we do not need here beyond continuity.  Strict
Jensen applies to every configuration because $V_{\Wset}$ is not almost everywhere
constant, so the attained minimum is strictly greater than one.
\end{proof}

The endpoint limits also select definite variational problems at the level
of minimizing configurations.  The following statement is one-sided: it
locates every cluster point of finite-exponent minimizers, but it does not
assert that every endpoint minimizer is obtained along such a family.

\begin{proposition}[Selection of endpoint minimizers]
\label{prop:theta-minimizer-endpoint-selection}
Fix a positive integer $N$, equip $\Xsurf^N$ with any product metric, and set
\begin{equation}\label{eq:quadratic-green-potential-energy}
 \mathcal E_N(\Wset):=\int_{\Xsurf}V_{\Wset}^2\,\diff\pmeas,
 \qquad e_N:=\min_{\Wset\in \Xsurf^N}\mathcal E_N(\Wset).
\end{equation}
Let
\[
 \mathfrak M_{N,\beta}:=\operatorname*{argmin}_{\Wset\in \Xsurf^N}\calB_{\beta}(\Wset;\Xsurf),
 \quad
 \mathfrak M_N^{(2)}:=\operatorname*{argmin}_{\Wset\in \Xsurf^N}\mathcal E_N(\Wset),
 \quad
 \mathfrak M_{N,\infty}:=\operatorname*{argmin}_{\Wset\in \Xsurf^N}
      \|Y_{\Wset}\|_{\infty}.
\]
Write the power means
\[
 F_{\beta}(\Wset):=\left(\int_{\Xsurf}Y_{\Wset}^\beta\,\diff\pmeas\right)^{1/\beta},
 \qquad F_{\infty}(\Wset):=\max_{\Xsurf}Y_{\Wset},
\]
and the uniform endpoint defect
\begin{equation}\label{eq:large-beta-uniform-defect}
 \omega_N(\beta):=\sup_{\Wset\in \Xsurf^N}
       \bigl(F_{\infty}(\Wset)-F_{\beta}(\Wset)\bigr).
\end{equation}
These sets are nonempty and compact, $e_N>0$, and
$\omega_N(\beta)\to0$ as $\beta\to\infty$.  As $\beta\downarrow0$,
\begin{align}
 \Theta_{N,\beta}(\Xsurf)
 &=1+\frac{\beta^2}{2}e_N+O_{\Xsurf,N}(\beta^3),
 \label{eq:theta-small-beta-expansion}\\
 \log\!\left(\Theta_{N,\beta}(\Xsurf)^{1/\beta}\right)
 &=\frac{\beta}{2}e_N+O_{\Xsurf,N}(\beta^2).
 \label{eq:theta-root-small-beta-expansion}
\end{align}
The objective gaps of finite-exponent minimizers satisfy
\begin{align}
 \sup_{\Wset\in\mathfrak M_{N,\beta}}
  \bigl(\mathcal E_N(\Wset)-e_N\bigr)
 &\leq C_{\Xsurf,N}\beta,
 &&0<\beta\leq1,
 \label{eq:small-beta-objective-gap}\\
 \sup_{\Wset\in\mathfrak M_{N,\beta}}
  \bigl(F_{\infty}(\Wset)-\Theta_{N,\infty}(\Xsurf)\bigr)
 &\leq\omega_N(\beta).
 \label{eq:large-beta-objective-gap}
\end{align}
Moreover,
\begin{align}
 \lim_{\beta\downarrow0}
 \sup_{\Wset\in\mathfrak M_{N,\beta}}
 \dist_{\Xsurf^N}(\Wset,\mathfrak M_N^{(2)})&=0,
 \label{eq:small-beta-minimizer-selection}\\
 \lim_{\beta\to\infty}
 \sup_{\Wset\in\mathfrak M_{N,\beta}}
 \dist_{\Xsurf^N}(\Wset,\mathfrak M_{N,\infty})&=0.
 \label{eq:large-beta-minimizer-selection}
\end{align}
In particular, every cluster point of minimizers as $\beta\downarrow0$
minimizes the quadratic Green-potential energy, and every cluster point as
$\beta\to\infty$ minimizes the pointwise endpoint.
\end{proposition}

\begin{proof}
The uniform local logarithmic model for the Green kernel and compactness of
$\Xsurf$ give
\[
 C_3:=\sup_{w\in \Xsurf}\|U_{\Xsurf}(\,\cdot\,,w)\|_{L^3(\Xsurf)}<\infty.
\]
Also $M_{\Xsurf}:=\sup_{\Xsurf\times \Xsurf}U_{\Xsurf}<\infty$.  Hence, uniformly for $\Wset\in \Xsurf^N$,
\[
 \|V_{\Wset}\|_{L^3}\leq NC_3,
 \qquad V_{\Wset}\leq NM_{\Xsurf}.
\]
For real $t$ one has the Taylor remainder bound
\[
 \left|e^t-1-t-\frac{t^2}{2}\right|
 \leq\frac{|t|^3}{6}e^{t_+},
 \qquad t_+:=\max\{t,0\}.
\]
Apply this with $t=\beta V_{\Wset}$, use $\int_{\Xsurf}V_{\Wset}\,\diff\pmeas=0$, and restrict first
to $0<\beta\leq1$.  The preceding uniform bounds yield
\begin{equation}\label{eq:uniform-small-beta-Taylor}
 \sup_{\Wset\in \Xsurf^N}
 \left|\calB_{\beta}(\Wset;\Xsurf)-1-\frac{\beta^2}{2}\mathcal E_N(\Wset)\right|
 \leq C_{\Xsurf,N}\beta^3.
\end{equation}
Consequently
\[
 G_{\beta}(\Wset):=\frac{2}{\beta^2}\bigl(\calB_{\beta}(\Wset;\Xsurf)-1\bigr)
 \longrightarrow\mathcal E_N(\Wset)
\]
uniformly on $\Xsurf^N$.  Each $G_{\beta}$ is continuous, so $\mathcal E_N$ is
continuous; its minimum is therefore attained.  Moreover,
$\mathcal E_N(\Wset)=0$ would make $V_{\Wset}=0$ almost everywhere, which is
incompatible with its logarithmic singularities.  Compact attainment then
gives $e_N>0$.  Taking infima in
\eqref{eq:uniform-small-beta-Taylor} proves
\eqref{eq:theta-small-beta-expansion}.  Taking the logarithm and dividing
by $\beta$ gives \eqref{eq:theta-root-small-beta-expansion}.

Fix $\Wset_{\beta}\in\mathfrak M_{N,\beta}$ and
$\Wset^{(2)}\in\mathfrak M_N^{(2)}$.  The lower estimate in
\eqref{eq:uniform-small-beta-Taylor} at $\Wset_{\beta}$, the upper estimate at
$\Wset^{(2)}$, and the minimality of $\Wset_{\beta}$ give
\[
 \frac{\beta^2}{2}\bigl(\mathcal E_N(\Wset_{\beta})-e_N\bigr)
 \leq2C_{\Xsurf,N}\beta^3.
\]
After adjusting the constant, this proves
\eqref{eq:small-beta-objective-gap}.  The minimizers of $G_{\beta}$ are
exactly those of $\calB_{\beta}$.  Uniform convergence on the compact space
$\Xsurf^N$ now proves
\eqref{eq:small-beta-minimizer-selection}.  Indeed, otherwise there would
be $\varepsilon>0$, a sequence $\beta_k\downarrow0$, and
$\Wset_k\in\mathfrak M_{N,\beta_k}$ staying at least $\varepsilon$ from
$\mathfrak M_N^{(2)}$.  After extraction, $\Wset_k\to \Wset_*$.  Minimality and
uniform convergence give
\[
 \mathcal E_N(\Wset_*)\leq\mathcal E_N(\Wset)
 \qquad\text{for every }\Wset\in \Xsurf^N,
\]
contradicting the distance condition.

For the other endpoint, the proof of \Cref{prop:theta-beta-endpoints} shows
that $F_{\beta}\to F_{\infty}$ uniformly on $\Xsurf^N$ as $\beta\to\infty$; hence
$\omega_N(\beta)\to0$.  The functions $\calB_{\beta}$ and $F_{\beta}$ have the
same minimizers.  If
$\Wset_{\infty}\in\mathfrak M_{N,\infty}$, then
\[
 F_{\infty}(\Wset_{\beta})
 \leq F_{\beta}(\Wset_{\beta})+\omega_N(\beta)
 \leq F_{\beta}(\Wset_{\infty})+\omega_N(\beta)
 \leq\Theta_{N,\infty}(\Xsurf)+\omega_N(\beta).
\]
This proves \eqref{eq:large-beta-objective-gap}.  Repeating the preceding
compactness argument with the uniform convergence of $F_{\beta}$ proves
\eqref{eq:large-beta-minimizer-selection} and completes the proof.
\end{proof}

\begin{proposition}[Gibbs envelope and exponent derivatives]
\label{prop:theta-Gibbs-envelope}
For a fixed positive integer $N$, consider
\begin{equation}\label{eq:theta-Gibbs-envelope}
 \mathcal F_N(\beta)
 :=\log\!\left(\Theta_{N,\beta}(\Xsurf)^{1/\beta}\right)
 =\min_{\Wset\in \Xsurf^N}\frac1\beta\log \calB_{\beta}(\Wset;\Xsurf),
 \qquad \beta>0.
\end{equation}
As above, write
$\mathfrak M_{N,\beta}=\operatorname*{argmin}_{\Wset\in \Xsurf^N}\calB_{\beta}(\Wset;\Xsurf)$.
For $\Wset\in \Xsurf^N$, let $\mathsf G_{\beta,\Wset}$ be the Gibbs probability measure
\begin{equation}\label{eq:theta-Gibbs-measure}
 d\mathsf G_{\beta,\Wset}
 :=\varrho_{\beta,\Wset}\,\diff\pmeas,
 \qquad
 \varrho_{\beta,\Wset}
 :=\frac{Y_{\Wset}^\beta}{\calB_{\beta}(\Wset;\Xsurf)},
\end{equation}
and use the convention $0\log0=0$ in the relative entropy
\[
 \operatorname{Ent}_{\pmeas}(\mathsf G_{\beta,\Wset})
 :=\int_{\Xsurf}\varrho_{\beta,\Wset}\log\varrho_{\beta,\Wset}\,\diff\pmeas.
\]
Then $\mathcal F_N$ is locally Lipschitz on $(0,\infty)$.  Its finite right
and left derivatives exist at every $\beta>0$ and satisfy
\begin{align}
 D^+\mathcal F_N(\beta)
 &:=\lim_{h\downarrow0}
   \frac{\mathcal F_N(\beta+h)-\mathcal F_N(\beta)}{h}
 =\frac1{\beta^2}
   \min_{\Wset\in\mathfrak M_{N,\beta}}
   \operatorname{Ent}_{\pmeas}(\mathsf G_{\beta,\Wset}),
 \label{eq:theta-Gibbs-right-derivative}\\
 D^-\mathcal F_N(\beta)
 &:=\lim_{h\downarrow0}
   \frac{\mathcal F_N(\beta)-\mathcal F_N(\beta-h)}{h}
 =\frac1{\beta^2}
   \max_{\Wset\in\mathfrak M_{N,\beta}}
   \operatorname{Ent}_{\pmeas}(\mathsf G_{\beta,\Wset}).
 \label{eq:theta-Gibbs-left-derivative}
\end{align}
Every entropy appearing here is strictly positive.  Consequently,
\[
 0<D^+\mathcal F_N(\beta)\leq D^-\mathcal F_N(\beta).
\]
The function $\mathcal F_N$ is differentiable at $\beta$ if and only if all
configurations in $\mathfrak M_{N,\beta}$ have the same Gibbs relative
entropy; at such a point,
\[
 \mathcal F_N'(\beta)
 =\frac{\operatorname{Ent}_{\pmeas}(\mathsf G_{\beta,\Wset})}{\beta^2},
 \qquad \Wset\in\mathfrak M_{N,\beta}.
\]
In particular, $\mathcal F_N$ is differentiable almost everywhere.  Moreover,
the map
\[
 \beta\longmapsto\Theta_{N,\beta}(\Xsurf)^{1/\beta}
\]
is locally Lipschitz and defines a strictly increasing homeomorphism from
$(0,\infty)$ onto $(1,\Theta_{N,\infty}(\Xsurf))$.
\end{proposition}

\begin{proof}
The proof has three steps.  We first justify differentiation with respect to
$\beta$ uniformly over the compact configuration space, including at
colliding nodes.  We then differentiate the minimum envelope without
choosing a continuous minimizing branch.  Finally, relative-entropy
positivity and the already established endpoint limits give the remaining
claims.

Write
\[
 \ell(\beta,\Wset):=\frac1\beta\log \calB_{\beta}(\Wset;\Xsurf),
 \qquad
 R_N:=\max_{(z,\Wset)\in \Xsurf\times \Xsurf^N}Y_{\Wset}(z)<\infty.
\]
The function $(z,\Wset)\mapsto Y_{\Wset}(z)$ is continuous and nonnegative on the
compact space $\Xsurf\times \Xsurf^N$.  Fix a compact interval
$[a,b]\Subset(0,\infty)$.  The maps
\[
 (\beta,y)\longmapsto y^\beta,
\qquad
 (\beta,y)\longmapsto y^\beta\log y
\]
are continuous on $[a,b]\times[0,R_N]$ when the second is assigned the
value zero at $y=0$.  Indeed, for $0<y\leq1$,
\[
 \sup_{\beta\in[a,b]}|y^\beta\log y|
 \leq y^a|\log y|\longrightarrow0
 \qquad (y\downarrow0).
\]
It follows that
\[
 \calB_{\beta}(\Wset;\Xsurf)=\int_{\Xsurf}Y_{\Wset}^\beta\,\diff\pmeas,
 \qquad
 \partial_{\beta} \calB_{\beta}(\Wset;\Xsurf)
 =\int_{\Xsurf}Y_{\Wset}^\beta\log Y_{\Wset}\,\diff\pmeas
\]
and that both functions on the right are jointly continuous in
$(\beta,\Wset)\in[a,b]\times \Xsurf^N$.  Since $\calB_{\beta}(\Wset;\Xsurf)\geq1$ by Jensen's
inequality, $\ell$ is continuously differentiable in $\beta$, uniformly
over $\Wset$ on compact $\beta$-intervals.  Direct calculation gives
\begin{align}
 \partial_{\beta}\ell(\beta,\Wset)
 &=-\frac{\log \calB_{\beta}(\Wset;\Xsurf)}{\beta^2}
   +\frac1\beta
    \frac{\int_{\Xsurf}Y_{\Wset}^\beta\log Y_{\Wset}\,\diff\pmeas}
         {\calB_{\beta}(\Wset;\Xsurf)}\notag\\
 &=\frac1{\beta^2}
   \operatorname{Ent}_{\pmeas}(\mathsf G_{\beta,\Wset}).
 \label{eq:Gibbs-entropy-slope}
\end{align}
This derivative is jointly continuous.  It is therefore uniformly bounded
on $[a,b]\times \Xsurf^N$, so the functions
$\ell(\,\cdot\,,\Wset)$ are equi-Lipschitz there.  Their minimum
$\mathcal F_N$ is consequently locally Lipschitz.

We next prove the one-sided formulas without assuming a differentiable
choice of minimizer.  Fix $\beta>0$.  For every
$\Wset\in\mathfrak M_{N,\beta}$ and $h>0$,
\[
 \mathcal F_N(\beta+h)-\mathcal F_N(\beta)
 \leq\ell(\beta+h,\Wset)-\ell(\beta,\Wset).
\]
Taking the upper limit, then minimizing over the active set, and using
\eqref{eq:Gibbs-entropy-slope} gives the upper bound in
\eqref{eq:theta-Gibbs-right-derivative}.  Conversely, choose
$h_k\downarrow0$ so that the right difference quotient converges to its
lower limit, and choose
$\Wset_k\in\mathfrak M_{N,\beta+h_k}$.  After passing to a subsequence,
$\Wset_k\to \Wset_*$.  Continuity of $\mathcal F_N$ and joint continuity of $\ell$
give
\[
 \ell(\beta,\Wset_*)=\mathcal F_N(\beta),
\]
so $\Wset_*\in\mathfrak M_{N,\beta}$.  Moreover,
\begin{align*}
 \frac{\mathcal F_N(\beta+h_k)-\mathcal F_N(\beta)}{h_k}
 &\geq
 \frac{\ell(\beta+h_k,\Wset_k)-\ell(\beta,\Wset_k)}{h_k}\\
 &=\int_0^1
   \partial_{\beta}\ell(\beta+t h_k,\Wset_k)\,dt
 \longrightarrow\partial_{\beta}\ell(\beta,\Wset_*).
\end{align*}
The limiting slope is at least the minimum over the active set.  This proves
the reverse lower bound and hence the genuine right derivative in
\eqref{eq:theta-Gibbs-right-derivative}.

For the left derivative, testing first with every
$\Wset\in\mathfrak M_{N,\beta}$ gives
\[
 \liminf_{h\downarrow0}
 \frac{\mathcal F_N(\beta)-\mathcal F_N(\beta-h)}{h}
 \geq
 \max_{\Wset\in\mathfrak M_{N,\beta}}
 \partial_{\beta}\ell(\beta,\Wset).
\]
For the reverse inequality, choose $h_k\downarrow0$ so that the left
difference quotient converges to its upper limit, take
$\Wset_k\in\mathfrak M_{N,\beta-h_k}$, and extract a convergent subsequence.  The
limit, say $\Wset_k\to \Wset_*$, again belongs to
$\mathfrak M_{N,\beta}$.  Indeed, continuity gives
$\ell(\beta,\Wset_*)=\mathcal F_N(\beta)$.  Moreover,
\begin{align*}
 \frac{\mathcal F_N(\beta)-\mathcal F_N(\beta-h_k)}{h_k}
 &\leq
 \frac{\ell(\beta,\Wset_k)-\ell(\beta-h_k,\Wset_k)}{h_k}\\
 &=\int_0^1
   \partial_{\beta}\ell(\beta-h_k+t h_k,\Wset_k)\,dt
 \longrightarrow\partial_{\beta}\ell(\beta,\Wset_*).
\end{align*}
The last slope is at most the maximum over the active set.  This proves the
reverse upper bound and hence
\eqref{eq:theta-Gibbs-left-derivative}.

The relative entropy is nonnegative by the strict convexity of
$t\mapsto t\log t$ and the identity
$\int_{\Xsurf}\varrho_{\beta,\Wset}\,\diff\pmeas=1$.  Equality would force
$\varrho_{\beta,\Wset}=1$ almost everywhere.  Since this density is continuous
and $\pmeas$ has full support, it would then equal one everywhere.  This is
impossible because $Y_{\Wset}$, and hence $\varrho_{\beta,\Wset}$, vanishes at every
node.  Thus every active entropy is positive.  It depends continuously on
$\Wset$, while $\mathfrak M_{N,\beta}$ is compact, so its minimum on the active
set is strictly positive.

The differentiability criterion follows by comparing the minimum and
maximum in \eqref{eq:theta-Gibbs-right-derivative}--%
\eqref{eq:theta-Gibbs-left-derivative}; almost-everywhere differentiability
follows from local Lipschitz continuity.  Finally,
\Cref{prop:theta-beta-monotone,prop:theta-beta-endpoints} give strict
monotonicity and the two endpoint limits.  Together with local Lipschitz
continuity, they show that
$\beta\mapsto\Theta_{N,\beta}(\Xsurf)^{1/\beta}$ has the asserted homeomorphism
property.
\end{proof}

\begin{remark}[Scope of the exponent calculus]
All assertions in \Cref{prop:theta-Gibbs-envelope} are for one fixed $N$.
The proof gives no Lipschitz or entropy bound uniform in $N$, no uniqueness
of a minimizing configuration, and no continuous selection
$\beta\mapsto \Wset_{\beta}$.  Active minimizing branches may exchange, so the
envelope is asserted to be locally Lipschitz and almost everywhere
differentiable, not $C^1$.
\end{remark}

One further structural relation is available without any information about
minimizers: concatenating configurations of different cardinalities adds their
potentials, and H\"older separates the resulting exponentials.

\begin{proposition}[Concatenation inequality]
\label{prop:theta-concatenation}
For all positive integers $N,M$, every $\beta>0$, and conjugate exponents
$r,s>1$ with $r^{-1}+s^{-1}=1$,
\begin{equation}\label{eq:theta-concatenation}
 \Theta_{N+M,\beta}(\Xsurf)
 \leq\Theta_{N,r\beta}(\Xsurf)^{1/r}
       \Theta_{M,s\beta}(\Xsurf)^{1/s}.
\end{equation}
\end{proposition}

\begin{proof}
Given configurations $\Wset\in \Xsurf^N$ and $\Wset'\in \Xsurf^M$, concatenate them.  Their
Green potentials add.  Writing $V_{\Wset}=\sum_{w\in \Wset}U_{\Xsurf}(\,\cdot\,,w)$ and
$V_{\Wset'}=\sum_{w'\in \Wset'}U_{\Xsurf}(\,\cdot\,,w')$, one has pointwise
\[
 e^{\beta(V_{\Wset}+V_{\Wset'})}
 =\bigl(e^{r\beta V_{\Wset}}\bigr)^{1/r}
  \bigl(e^{s\beta V_{\Wset'}}\bigr)^{1/s}.
\]
H\"older's inequality therefore gives
\[
 \calB_{\beta}(\Wset\cup \Wset';\Xsurf)
 \leq \calB_{r\beta}(\Wset;\Xsurf)^{1/r}\calB_{s\beta}(\Wset';\Xsurf)^{1/s}.
\]
First take the infimum independently over $\Wset$ and $\Wset'$, and then use the
definition of $\Theta_{N+M,\beta}$.
\end{proof}

\subsection{An exact two-point model on the round sphere}
\label{subsec:sphere-two-point-model}

The two general variational problems admit a complete first nontrivial
fixed-cardinality solution.  Equip $\bP^1$ with the unitary-invariant
probability measure $\pmeas_{\rm FS}$.  For unit lifts
$\zeta,\omega\in\C^2$, set
\[
 d_{\rm ch}([\zeta],[\omega]):=|\det(\zeta,\omega)|,
 \qquad
 d_{\rm ch}([\zeta],[\omega])^2
 =1-|\langle\zeta,\omega\rangle|^2.
\]
Let $\mathfrak a(p)$ denote the spherical antipode of $p$, namely the
orthogonal complex line.  With the conventions of
\eqref{eq:green-normalization}, the normalized Green kernel is
\begin{equation}\label{eq:sphere-two-point-green}
 U_{\rm FS}(z,w)=\frac12+\log d_{\rm ch}(z,w).
\end{equation}
Indeed, $d_{\rm ch}(z,w)^2$ is uniformly distributed on $[0,1]$ for
$z\sim\pmeas_{\rm FS}$, so the logarithmic term has mean $-1/2$.

For $p,q\in\bP^1$, write
\begin{equation}\label{eq:sphere-two-point-s}
 s:=d_{\rm ch}(q,\mathfrak a(p))^2
   =1-d_{\rm ch}(p,q)^2\in[0,1].
\end{equation}
Thus $s=0$ exactly for antipodal pairs and $s=1$ exactly for collisions.

\begin{lemma}[A correlated complex-Gaussian moment]
\label{lem:sphere-correlated-gaussian}
Let $a>0$, and let $(X,Y)$ be a jointly proper centered complex Gaussian
vector with
\[
 \mathbb E|X|^2=\mathbb E|Y|^2=1,\qquad
 \mathbb E(X\overline{Y})=c,\qquad |c|^2=s\in[0,1].
\]
Here properness means $\mathbb E X^2=\mathbb E Y^2=\mathbb E(XY)=0$.
Then
\begin{equation}\label{eq:sphere-correlated-gaussian}
 \mathbb E\bigl(|X|^{2a}|Y|^{2a}\bigr)
 =\Gamma(1+a)^2\,{}_2F_1(-a,-a;1;s).
\end{equation}
\end{lemma}

\begin{proof}
For $0\le s<1$, rotate $Y$ so that $c=\sqrt s$.  If
$u=|X|^2$ and $v=|Y|^2$, their joint density is
\[
 \frac1{1-s}\exp\!\left(-\frac{u+v}{1-s}\right)
 I_0\!\left(\frac{2\sqrt{suv}}{1-s}\right),
 \qquad u,v>0,
\]
where
$I_0(t)=\sum_{k\ge0}(t^2/4)^k/(k!)^2$.  Tonelli's theorem and the gamma
integral give
\[
 \mathbb E(u^av^a)
 =\Gamma(1+a)^2(1-s)^{2a+1}
   {}_2F_1(1+a,1+a;1;s).
\]
Euler's transformation \cite[Eq.~15.8.1]{DLMF} gives
\[
 (1-s)^{2a+1}{}_2F_1(1+a,1+a;1;s)
 ={}_2F_1(-a,-a;1;s).
\]
If $s=1$, then $Y$ is a unit-modulus multiple of $X$; equivalently, one may use
Gauss' value at one \cite[Eq.~15.4.20]{DLMF}.
\end{proof}

\begin{theorem}[Exact two-point Green packing on $\bP^1$]
\label{thm:sphere-two-point-exact}
For every $\beta>0$, write $a=\beta/2$.  For $p,q\in\bP^1$ and $s$ as in
\eqref{eq:sphere-two-point-s},
\begin{equation}\label{eq:sphere-two-point-exact}
 \calB_{\beta}((p,q);\bP^1)
 =e^\beta\frac{\Gamma(1+a)^2}{\Gamma(2+2a)}
   {}_2F_1(-a,-a;1;s).
\end{equation}
Consequently:
\begin{enumerate}[label=\textup{(\roman*)}]
\item the minimizing pairs are exactly the antipodal pairs
$q=\mathfrak a(p)$, unique modulo $\mathrm{PSU}(2)$ and permutation;
\item
\begin{equation}\label{eq:sphere-two-point-theta}
 \Theta_{2,\beta}(\bP^1)
 =e^\beta\frac{\Gamma(1+\beta/2)^2}{\Gamma(2+\beta)};
\end{equation}
\item the collision value is
\begin{equation}\label{eq:sphere-two-point-collision}
 \calB_{\beta}((p,p);\bP^1)=\frac{e^\beta}{1+\beta};
\end{equation}
\item the global stability estimate
\begin{equation}\label{eq:sphere-two-point-stability}
 \frac{\beta^2}{4}s
 \le
 \frac{\calB_{\beta}((p,q);\bP^1)}{\Theta_{2,\beta}(\bP^1)}-1
 \le
 \left(
  \frac{\Gamma(1+\beta)}{\Gamma(1+\beta/2)^2}-1
 \right)s
\end{equation}
holds for every pair.
\end{enumerate}
\end{theorem}

\begin{proof}
The computation is a Gaussian one.  Writing a standard complex Gaussian vector as
radius times Fubini--Study direction makes the radius independent of the
direction, so the spherical integral becomes a ratio of Gaussian moments.  The two
linear forms vanishing at $p$ and $q$ then become correlated complex Gaussians,
and their squared correlation is exactly the parameter $s$.

Choose unit linear forms $\ell_p,\ell_q$ vanishing at $p,q$.  Let $G$ be a
standard complex Gaussian vector in $\C^2$ and write $G=R\Xi$.  The random
variables $R$ and $[\Xi]$ are independent, $[\Xi]\sim\pmeas_{\rm FS}$, and
$R^2$ has the gamma distribution of shape two.  Moreover,
\[
 d_{\rm ch}([\Xi],p)=|\ell_p(\Xi)|,
 \qquad
 d_{\rm ch}([\Xi],q)=|\ell_q(\Xi)|,
 \qquad
 |\langle\ell_p,\ell_q\rangle|^2=s.
\]
Thus \Cref{lem:sphere-correlated-gaussian} and
$\mathbb E(R^{4a})=\Gamma(2+2a)$ give
\[
 \int_{\bP^1}d_{\rm ch}(z,p)^{2a}d_{\rm ch}(z,q)^{2a}
 \,\diff\pmeas_{\rm FS}(z)
 =\frac{\Gamma(1+a)^2}{\Gamma(2+2a)}
  {}_2F_1(-a,-a;1;s).
\]
Together with \eqref{eq:sphere-two-point-green}, this is
\eqref{eq:sphere-two-point-exact}.

Now
\[
 {}_2F_1(-a,-a;1;s)
 =\sum_{k=0}^\infty\frac{(-a)_k^2}{(k!)^2}s^k.
\]
All coefficients are nonnegative and the coefficient of $s$ is $a^2>0$.
Hence the function is strictly increasing on $[0,1]$, which proves the
classification and \eqref{eq:sphere-two-point-theta}.  At $s=1$, Gauss'
value at one \cite[Eq.~15.4.20]{DLMF} gives
\[
 {}_2F_1(-a,-a;1;1)
 =\frac{\Gamma(1+2a)}{\Gamma(1+a)^2},
\]
and hence \eqref{eq:sphere-two-point-collision}.  The first nonconstant
series term gives the lower bound in
\eqref{eq:sphere-two-point-stability}; using $s^k\le s$ for $k\ge1$ and the
preceding value at one gives the upper bound.
\end{proof}

\begin{corollary}[Morse--Bott nondegeneracy]
\label{cor:sphere-two-point-morse-bott}
Fix an antipodal pair $(p,\mathfrak a(p))$ and keep $p$ fixed.  If the second
point is moved through angular distance $\varepsilon$ from
$\mathfrak a(p)$, then
\begin{equation}\label{eq:sphere-two-point-local-expansion}
 \calB_{\beta}((p,q_{\varepsilon});\bP^1)
 =\Theta_{2,\beta}(\bP^1)
 \left(1+\frac{\beta^2}{16}\varepsilon^2
       +O_{\beta}(\varepsilon^4)\right).
\end{equation}
The antipodal critical manifold is Morse--Bott.  For the product of the
unit-round metrics on $S^2\times S^2$, its two nonzero normal Hessian
eigenvalues are
\begin{equation}\label{eq:sphere-two-point-normal-hessian}
 \frac{\beta^2}{4}\Theta_{2,\beta}(\bP^1).
\end{equation}
\end{corollary}

\begin{proof}
Here $s=\sin^2(\varepsilon/2)=\varepsilon^2/4+O(\varepsilon^4)$, and the
linear coefficient of the hypergeometric factor is $(\beta/2)^2$.  This
proves \eqref{eq:sphere-two-point-local-expansion}.  A product-unit normal
vector to the antipodal manifold has the form $(v,v)/\sqrt2$ and produces
antipodal angular defect $\sqrt2\,t+O(t^3)$ at time $t$; substituting into
\eqref{eq:sphere-two-point-local-expansion} gives
\eqref{eq:sphere-two-point-normal-hessian}.
\end{proof}

\begin{corollary}[The two endpoint problems]
\label{cor:sphere-two-point-endpoints}
For arbitrary $p,q$ and $s$ as above,
\begin{align}
 \left\|e^{U_{\rm FS}(\cdot,p)+U_{\rm FS}(\cdot,q)}\right\|_{\infty}
 &=\frac e2(1+\sqrt s),
 \label{eq:sphere-two-point-infinity-pair}\\
 \int_{\bP^1}
 \bigl(U_{\rm FS}(z,p)+U_{\rm FS}(z,q)\bigr)^2\diff\pmeas_{\rm FS}(z)
 &=1-\frac{\pi^2}{12}+\frac12\operatorname{Li}_2(s).
 \label{eq:sphere-two-point-quadratic-pair}
\end{align}
Consequently, both endpoint minimizers are exactly the antipodal pairs, and
\begin{equation}\label{eq:sphere-two-point-endpoint-values}
 \Theta_{2,\infty}(\bP^1)=\frac e2,
 \qquad
 \min_{p,q}\int_{\bP^1}
 (U_{\rm FS}(z,p)+U_{\rm FS}(z,q))^2\diff\pmeas_{\rm FS}(z)
 =1-\frac{\pi^2}{12}.
\end{equation}
Moreover,
\begin{equation}\label{eq:sphere-two-point-large-beta}
 \Theta_{2,\beta}(\bP^1)
 =\left(\frac e2\right)^\beta
  \sqrt{\frac{\pi}{2\beta}}
  \left(1+O(\beta^{-1})\right)
 \qquad(\beta\to\infty).
\end{equation}
\end{corollary}

\begin{proof}
For \eqref{eq:sphere-two-point-infinity-pair}, choose phases so that
$|\langle\ell_p,\ell_q\rangle|=\sqrt s$.  For every unit $\zeta$,
\[
 |\ell_p(\zeta)\ell_q(\zeta)|
 \le\frac12\bigl(|\ell_p(\zeta)|^2+|\ell_q(\zeta)|^2\bigr)
 \le\frac12(1+\sqrt s),
\]
and equality holds at a top eigenvector of
$\ell_p^*\ell_p+\ell_q^*\ell_q$.

For the quadratic identity, compare the coefficient of $\beta^2$ in the
mean-zero expansion
\[
 \calB_{\beta}(\Wset)=1+\frac{\beta^2}{2}\int V_{\Wset}^2\diff\pmeas+O(\beta^3)
\]
with the exact formula.  One uses
\[
 \log\!\left(
 e^\beta\frac{\Gamma(1+\beta/2)^2}{\Gamma(2+\beta)}
 \right)
 =\left(\frac12-\frac{\pi^2}{24}\right)\beta^2+O(\beta^3)
\]
and
\[
 {}_2F_1(-\beta/2,-\beta/2;1;s)
 =1+\frac{\beta^2}{4}\operatorname{Li}_2(s)+O(\beta^3),
\]
uniformly for $s\in[0,1]$. To justify the latter uniform expansion, write
$a=\beta/2$. For $0<a\leq1/4$ and $k\geq1$,
\[
 \frac{(-a)_k^2}{(k!)^2}
 =\frac{a^2}{k^2}\prod_{j=1}^{k-1}\left(1-\frac aj\right)^2,
 \qquad
 0\leq1-\prod_{j=1}^{k-1}\left(1-\frac aj\right)^2
 \leq2a\sum_{j=1}^{k-1}\frac1j.
\]
The series of the resulting error bounds is summable, so the remainder is
$O(a^3)$ uniformly up to $s=1$. Finally,
\eqref{eq:sphere-two-point-large-beta} follows from the gamma duplication
formula \cite[Eq.~5.5.5]{DLMF} and Stirling's expansion
\cite[Eq.~5.11.1]{DLMF} applied to
\eqref{eq:sphere-two-point-theta}.
\end{proof}

\begin{remark}[Scope of the exact model]
\Cref{thm:sphere-two-point-exact} completely solves the $N=2$ round-sphere
case of \Cref{prob:minimizer-characterization} and the corresponding
fixed-$N$ branch of \Cref{prob:optimal-constants}.  It does not classify
minimizers for $N\ge3$ or identify the large-$N$ microscopic lattice.  It
does provide a strict normalization test for any future spherical or planar
cell problem.
\end{remark}

\subsection{The three-point sphere at the exponent two}
\label{subsec:sphere-three-two}

\Cref{prob:sphere-three-point} asks for $\Theta_{3,\beta}(\bP^1)$ and its
minimizers.  At the single exponent $\beta=2$ the problem can be solved
completely, and the solution identifies it with a classical extremal problem for
point configurations.  The mechanism is that at an even integer exponent the
integrand becomes the squared modulus of a holomorphic form, so the integral is a
Hermitian norm and the whole configuration dependence collapses onto the pairwise
overlaps.

We first record the orthogonality relation that makes the computation possible.

\begin{lemma}[Monomial orthogonality on $\bP^1$]\label{lem:sphere-basis-B}
With the normalization of \Cref{subsec:sphere-two-point-model}, for
$0\leq k,l\leq N$,
\begin{equation}\label{eq:sphere-monomial-orth}
 \int_{\bP^1}\zeta_0^{N-k}\zeta_1^k
 \overline{\zeta_0^{N-l}\zeta_1^l}\,\diff\pmeas_{\rm FS}
 =\delta_{kl}\,\frac1{(N+1)\binom Nk},
\end{equation}
where $\zeta$ denotes a unit representative and the integrand is evaluated on the
Hopf fibre, so that the expression descends to $\bP^1$.
\end{lemma}

\begin{proof}
The angular integration kills the off-diagonal terms, and the radial one is a beta
integral.

Work in the affine chart, where
\[
 \zeta=\frac{(1,z)}{\sqrt{1+|z|^2}},
 \qquad
 \diff\pmeas_{\rm FS}=\frac{dA(z)}{\pi(1+|z|^2)^2},
\]
so that the integrand becomes $z^k\bar z^{\,l}(1+|z|^2)^{-N}$.  Passing to polar coordinates, the angular
integral vanishes unless $k=l$, which gives the Kronecker delta.  For $k=l$ the
substitution $t=|z|^2$ turns the remaining integral into the Euler beta integral
\[
 \int_0^\infty\frac{t^k}{(1+t)^{N+2}}\,dt
 =\frac{k!\,(N-k)!}{(N+1)!}
 =\frac1{(N+1)\binom Nk},
\]
which is \eqref{eq:sphere-monomial-orth}.
\end{proof}

The orthogonality relation makes the three-point energy explicit.  We
express it in the natural invariant of a pair of points on $\bP^1$.

For unit representatives $\zeta,\omega\in\C^2$ of points of $\bP^1$ write
\begin{equation}\label{eq:overlap-def}
 s(\zeta,\omega):=|\langle\zeta,\omega\rangle|^2
 =1-d_{\rm ch}([\zeta],[\omega])^2 ,
\end{equation}
so that $s=0$ for antipodal points and $s=1$ for coincident ones.

\begin{theorem}[Exact three-point energy at $\beta=2$]
\label{thm:sphere-three-point-beta2}
With the conventions of \Cref{subsec:sphere-two-point-model}, for every
$\Wset=(w_1,w_2,w_3)\in(\bP^1)^3$ with unit representatives $\zeta_1,\zeta_2,\zeta_3$,
\begin{equation}\label{eq:sphere-three-point-beta2}
 \calB_2(\Wset;\bP^1)=\frac{e^3}{12}\sum_{1\leq i<j\leq3}s(\zeta_i,\zeta_j).
\end{equation}
Consequently
\begin{equation}\label{eq:theta-three-two}
 \Theta_{3,2}(\bP^1)=\frac{e^3}{16},
\end{equation}
and the minimizers are exactly the equilateral triangles inscribed in a great
circle, unique modulo $\mathrm{PSU}(2)$ and permutation.
\end{theorem}

\begin{proof}
There are three steps.  At the exponent $2$ the integrand is the squared modulus of
a holomorphic cubic, so $\calB_2$ becomes a Bombieri--Weyl norm.  That norm equals a
permanent of the Gram matrix of overlaps, which simplifies for three vectors in
$\C^2$ because the Gram matrix is then singular.  The resulting expression is
minimized by a classical distance-maximization on the sphere.

Let $\ell_j(\zeta):=\det(\zeta,\zeta_j)$, a linear form vanishing precisely at
$w_j$, and write $P:=\ell_1\ell_2\ell_3$.  For a unit representative $\zeta$ of $z$
one has $d_{\rm ch}(z,w_j)=|\ell_j(\zeta)|$, so
$U_{\rm FS}(z,w_j)=\tfrac12+\log|\ell_j(\zeta)|$ and therefore
\[
 \exp\Bigl(2\sum_{j=1}^3U_{\rm FS}(z,w_j)\Bigr)=e^3|P(\zeta)|^2 .
\]
Hence $\calB_2(\Wset;\bP^1)=e^3\int_{\bP^1}|P|^2\,\diff\pmeas_{\rm FS}$, and
\Cref{lem:sphere-basis-B} with $N=3$ evaluates that integral coefficientwise:
writing $P=\sum_{k=0}^3a_k\zeta_0^{3-k}\zeta_1^k$,
\begin{equation}\label{eq:three-point-norm-sum}
 \calB_2(\Wset;\bP^1)=\frac{e^3}{4}\sum_{k=0}^3\frac{|a_k|^2}{\binom3k}
 =\frac{e^3}{4}\,\|P\|_{\mathrm{BW}}^2 .
\end{equation}

We evaluate $\|P\|_{\mathrm{BW}}^2$ exactly, by
\Cref{lem:bw-permanent} below: for any $N$ and any unit vectors
$\zeta_1,\ldots,\zeta_N$,
\begin{equation}\label{eq:bw-perm-three}
 \|\ell_1\cdots\ell_N\|_{\mathrm{BW}}^2=\frac1{N!}\operatorname{per}G,
 \qquad G_{ij}:=\langle\zeta_i,\zeta_j\rangle .
\end{equation}
For $N=3$ the permanent expands over the identity, the three transpositions and
the two $3$-cycles, and $G_{jj}=1$, so
\begin{equation}\label{eq:perm-three-expand}
 \operatorname{per}G
 =1+\sum_{i<j}|G_{ij}|^2+2\operatorname{Re}(G_{12}G_{23}G_{31}).
\end{equation}
The triple term is not independent.  The $\zeta_j$ lie in $\C^2$, so $G$ has rank
at most $2$ and $\det G=0$; expanding that determinant in the same way gives
\begin{equation}\label{eq:det-three-expand}
 0=\det G=1-\sum_{i<j}|G_{ij}|^2+2\operatorname{Re}(G_{12}G_{23}G_{31}),
\end{equation}
whence $2\operatorname{Re}(G_{12}G_{23}G_{31})=\sum_{i<j}|G_{ij}|^2-1$.
Substituting into \eqref{eq:perm-three-expand} and recalling
$s(\zeta_i,\zeta_j)=|G_{ij}|^2$,
\[
 \operatorname{per}G=2\sum_{i<j}s(\zeta_i,\zeta_j),
 \qquad\text{so}\qquad
 \|P\|_{\mathrm{BW}}^2=\frac13\sum_{i<j}s(\zeta_i,\zeta_j),
\]
and \eqref{eq:three-point-norm-sum} gives
\eqref{eq:sphere-three-point-beta2}.

It remains to minimize.  By \eqref{eq:overlap-def},
\[
 \text{minimizing }\ \sum_{i<j}s(\zeta_i,\zeta_j)
 \quad\Longleftrightarrow\quad
 \text{maximizing }\ \sum_{i<j}d_{\rm ch}(w_i,w_j)^2 .
\]
Let $x_1,x_2,x_3\in S^2$ be the unit vectors
corresponding to $w_1,w_2,w_3$ under the Hopf identification, so that
\begin{equation}\label{eq:chordal-euclid}
 d_{\rm ch}(w_i,w_j)^2=\frac{1-x_i\cdot x_j}{2}=\frac{|x_i-x_j|^2}{4}.
\end{equation}
Then
\[
 \sum_{i<j}|x_i-x_j|^2
 =3\sum_i|x_i|^2-\Bigl|\sum_ix_i\Bigr|^2
 =9-\Bigl|\sum_ix_i\Bigr|^2\leq9,
\]
with equality exactly when $x_1+x_2+x_3=0$, that is, when the three unit vectors
form an equilateral triangle in a plane through the origin.  By
\eqref{eq:chordal-euclid} the maximum of $\sum_{i<j}d_{\rm ch}^2$ is therefore
$9/4$, the minimum of $\sum_{i<j}s$ is $3-9/4=3/4$, and
\eqref{eq:sphere-three-point-beta2} gives
$\Theta_{3,2}=\tfrac{e^3}{12}\cdot\tfrac34=e^3/16$.

Finally we check the uniqueness clause.  The minimizers are exactly the ordered
triples $(x_1,x_2,x_3)$ of unit vectors with $x_1+x_2+x_3=0$.  Pairing that
relation with $x_i$ gives $x_i\cdot x_j=-\tfrac12$ for all $i\neq j$.  Recall that
$\mathrm{PSU}(2)$ acts on $S^2$ as the full rotation group $\mathrm{SO}(3)$; we
claim it is transitive on such ordered triples, which are in general spanned by
different planes.

Attach to a minimizing triple the orthonormal frame
\[
 f_1:=x_1,
 \qquad
 f_2:=\frac{x_2-(x_2\cdot x_1)x_1}{|x_2-(x_2\cdot x_1)x_1|},
 \qquad
 f_3:=f_1\times f_2 .
\]
This is well defined because $x_2\cdot x_1=-\tfrac12$, so $x_2$ is not parallel to
$x_1$, and $(f_1,f_2,f_3)$ is positively oriented by construction.  Moreover the
triple is recovered from its frame by universal formulas: $x_1=f_1$, and
\[
 x_2=-\tfrac12f_1+\tfrac{\sqrt3}2f_2,
 \qquad
 x_3=-x_1-x_2=-\tfrac12f_1-\tfrac{\sqrt3}2f_2,
\]
since $x_2$ has component $-\tfrac12$ along $f_1$, unit length, and nonnegative
$f_2$-component by the choice of $f_2$.  The coefficients involve no data beyond
the Gram matrix, which is the same for all minimizing triples.

Now let $(x_j)$ and $(x_j')$ be two minimizing triples with frames $(f_i)$ and
$(f_i')$.  Both frames are positively oriented orthonormal bases of $\R^3$, so
there is a unique $R\in\mathrm{SO}(3)$ with $Rf_i=f_i'$ for $i=1,2,3$.  Applying
$R$ to the displayed formulas gives $Rx_j=x_j'$ for $j=1,2,3$.  Hence
$\mathrm{PSU}(2)$ is transitive on ordered minimizing triples and the minimizing
set is a single orbit; the permutation freedom in the statement reflects only the
labelling of $\Wset$ as an ordered configuration.
\end{proof}

\begin{lemma}[Bombieri--Weyl norm of a product of linear forms]
\label{lem:bw-permanent}
Let $\zeta_1,\ldots,\zeta_N\in\C^2$ and $\ell_j(\zeta)=\det(\zeta,\zeta_j)$.
Then \eqref{eq:bw-perm-three} holds, where $\operatorname{per}$ denotes the
permanent and $G_{ij}=\langle\zeta_i,\zeta_j\rangle$.
\end{lemma}

\begin{proof}
On degree-$N$ forms the Bombieri--Weyl inner product is $1/N!$ times the
unnormalized Fischer inner product.  Powers of linear forms satisfy
\[
 \langle\ell_u^N,\ell_v^N\rangle_{\mathrm{BW}}=\langle u,v\rangle^N,
 \qquad u,v\in\C^2 .
\]
Polarizing this identity in each of its $N$ slots yields
\[
 \langle\ell_{u_1}\cdots\ell_{u_N},\ell_{v_1}\cdots\ell_{v_N}\rangle_{\mathrm{BW}}
 =\frac1{N!}\sum_{\pi\in\mathfrak S_N}\prod_{j=1}^N\langle u_j,v_{\pi(j)}\rangle .
\]
Taking $u_j=v_j=\zeta_j$ turns the right-hand side into
$(N!)^{-1}\sum_{\pi}\prod_jG_{j\pi(j)}$, which is
$(N!)^{-1}\operatorname{per}G$.
\end{proof}

\begin{corollary}[Morse--Bott structure at three points and exponent two]
\label{cor:sphere-three-point-morse-bott}
At $N=3$ and $\beta=2$, every critical configuration with three distinct
points is an equilateral great-circle triangle. The minimizing configurations
form one smooth $\mathrm{PSU}(2)$ orbit, and the Hessian of $\calB_2$ is
positive definite in directions normal to that orbit. In particular, the
three-point minimum is Morse--Bott.
\end{corollary}

\begin{proof}
Under the Hopf identification, write the three points as unit vectors
$x_1,x_2,x_3\in S^2$ and set $S=x_1+x_2+x_3$. The identity
$s(\zeta_i,\zeta_j)=(1+x_i\cdot x_j)/2$ and
\eqref{eq:sphere-three-point-beta2} give the equivalent exact formula
\[
 \calB_2(x_1,x_2,x_3)=\frac{\e^3}{48}(3+|S|^2).
\]
Thus the gradient with respect to $x_j$, projected onto $T_{x_j}S^2$, is
$\e^3( S-(S\cdot x_j)x_j)/24$. If all three gradients vanish and
$S\neq0$, every $x_j$ belongs to $\{S/|S|,-S/|S|\}$, so two of the three
points must coincide. Hence a critical configuration with distinct points
has $S=0$, and is minimizing by the displayed formula.

At a minimum, a tangent velocity $v=(v_1,v_2,v_3)$ satisfies
\[
 \operatorname{Hess}\calB_2(v,v)
 =\frac{\e^3}{24}\left|v_1+v_2+v_3\right|^2.
\]
The map $\bigoplus_jT_{x_j}S^2\to\R^3$ taking $v$ to
$\sum_jv_j$ is surjective, since a vector orthogonal to its range would
have to be parallel to all three noncollinear $x_j$. Its kernel therefore
has dimension three. Infinitesimal simultaneous rotations lie in this
kernel because $\sum_jx_j=0$. They also span a three-dimensional space,
as the stabilizer of an ordered equilateral triple is discrete. Thus the
Hessian kernel is precisely the tangent space to the rotation orbit, and
the Hessian is positive definite on its normal space.
\end{proof}

\begin{remark}[Why this argument is confined to $\beta=2$ and $N\leq3$]
\label{rem:three-point-scope}
Two special features were used, and both fail in general.

The exponent had to be an even integer for $e^{\beta\sum_jU}$ to be the squared
modulus of a \emph{holomorphic} form, which is what turns $\calB_{\beta}$ into a
Hermitian norm.  For non-even $\beta$, this polynomial-norm representation is unavailable;
the two-point case instead has the hypergeometric formula of
\Cref{thm:sphere-two-point-exact}.

The cardinality had to satisfy $N\leq3$ for the squared norm to be
\emph{affine} in the overlaps. For $N\geq4$, consider $k$ points at one pole
and $N-k$ at its antipode. Their squared norm and sum of overlaps are
\[
 \|P_{\Wset}\|_{\mathrm{BW}}^2=\binom Nk^{-1},
 \qquad \sum_{i<j}s_{ij}=\binom N2-k(N-k).
\]
If an affine formula existed, averaging it over permutations would make all
overlap coefficients equal. The cases $k=0,1$ would then force the value
$(4-N)/N$ at $k=2$, contrary to the positive value $\binom N2^{-1}$.
Thus this affine reduction to a distance-maximization problem is special to
$N=2,3$. The argument does not settle \Cref{prob:sphere-three-point} for
$\beta\neq2$.
\end{remark}

Both restrictions suggest where to push next.  Fixing $\beta=2$ and raising $N$
keeps the Hermitian structure and loses only the affine dependence, which turns
the question into a finite-dimensional one.

\begin{problem}[Even exponents and higher cardinality]
\label{prob:even-exponent-tower}
Determine $\Theta_{N,2}(\bP^1)$ for $N\geq4$.  By
\Cref{thm:even-exponent-bombieri} this is the problem of minimizing the
Bombieri--Weyl norm of a product of $N$ unit linear forms, a finite-dimensional
algebraic optimization rather than an integral one.  Three questions, in
increasing order of difficulty.

First, prove \eqref{eq:even-exponent-balance}: that the minimizers are balanced.
As \Cref{rem:even-exponent-balance} explains, criticality is not enough, since the
coincident configuration is critical with maximal moment; a proof must use
minimality.

Second, determine whether the candidate configurations in
\Cref{rem:even-exponent-platonic} are globally minimizing for $4\leq N\leq8$.
For $2\leq N\leq7$ their squared norms are the reciprocals of
$2,4,9,20,48,112$; at $N=8$ the square-antiprism family has an algebraically
specified optimal parameter.  Is there a closed form for the unrestricted
minimum at general $N$, and does the combinatorial type of a minimizer stabilize?

Third, treat general even $\beta=2m$.  The same computation writes $\calB_{2m}$ as
$e^{mN}$ times a Hermitian norm on forms of degree $mN$, namely the
$\pmeas_{\rm FS}$-norm of $P_{\Wset}^m$, so the problem becomes minimization of
$\|P_{\Wset}^m\|$ over unit-root data.  Since Bombieri's inequality degrades under
powers, it is not clear whether the minimizers for different $m$ agree; deciding
that would in particular say whether the minimizing configuration depends on the
exponent, which is \Cref{prob:sphere-beta-independence}.
\end{problem}

\subsection{Even exponents: reduction to a Bombieri norm}
\label{subsec:even-exponent}

The mechanism behind \Cref{thm:sphere-three-point-beta2} does not need $N=3$.  At
$\beta=2$ and any cardinality the energy is a Hermitian norm, and identifying
which norm turns the variational problem into a finite-dimensional algebraic one.
The norm--integral correspondence is recorded in
\cite[Sections~1.3 and~2.1]{EtayoHedenmalmOrtega}.

\begin{theorem}[The exponent two is a Bombieri--Weyl norm]
\label{thm:even-exponent-bombieri}
Suppose $N\geq1$ and $\Wset=(w_1,\ldots,w_N)\in(\bP^1)^N$ has unit representatives
$\zeta_j$.  Associate to $\Wset$ the degree-$N$ form
\[
 P_{\Wset}(\zeta):=\prod_{j=1}^N\det(\zeta,\zeta_j).
\]
Let $\|\cdot\|_{\mathrm{BW}}$ denote the Bombieri--Weyl norm, given by
\[
 \Bigl\|\sum_{k=0}^Na_k\zeta_0^{N-k}\zeta_1^k\Bigr\|_{\mathrm{BW}}^2
 =\sum_{k=0}^N\frac{|a_k|^2}{\binom Nk}.
\]
Then
\begin{equation}\label{eq:even-exponent-identity}
 \calB_2(\Wset;\bP^1)=\frac{e^N}{N+1}\,\|P_{\Wset}\|_{\mathrm{BW}}^2 ,
\end{equation}
and consequently
\begin{equation}\label{eq:theta-N-two}
 \Theta_{N,2}(\bP^1)
 =\frac{e^N}{N+1}\,
 \min\bigl\{\|\ell_1\cdots\ell_N\|_{\mathrm{BW}}^2:
 \ \|\ell_j\|_{\mathrm{BW}}=1\bigr\},
\end{equation}
the minimum being over products of $N$ unit linear forms.
\end{theorem}

\begin{proof}
For a unit representative $\zeta$ of $z$ one has
$d_{\rm ch}(z,w_j)=|\det(\zeta,\zeta_j)|$, so
$U_{\rm FS}(z,w_j)=\tfrac12+\log|\det(\zeta,\zeta_j)|$ and
\[
 \exp\Bigl(2\sum_{j=1}^NU_{\rm FS}(z,w_j)\Bigr)=e^N|P_{\Wset}(\zeta)|^2 .
\]
Hence $\calB_2(\Wset;\bP^1)=e^N\int_{\bP^1}|P_{\Wset}|^2\,\diff\pmeas_{\rm FS}$, and
\Cref{lem:sphere-basis-B} evaluates that integral coefficientwise as
$(N+1)^{-1}\sum_k|a_k|^2/\binom Nk$, which is
\eqref{eq:even-exponent-identity}.  A single linear form
$\det(\,\cdot\,,\zeta_j)$ has $\|\cdot\|_{\mathrm{BW}}=1$ because $\zeta_j$ is a
unit vector, and every unit linear form arises this way; taking the infimum over
$\Wset$ gives \eqref{eq:theta-N-two}.
\end{proof}

Combining \eqref{eq:even-exponent-identity} with
\Cref{lem:bw-permanent} makes the objective completely explicit:
\begin{equation}\label{eq:even-exponent-permanent}
 \calB_2(\Wset;\bP^1)=\frac{e^N}{(N+1)\,N!}\operatorname{per}G,
 \qquad G_{ij}=\langle\zeta_i,\zeta_j\rangle ,
\end{equation}
so that minimizing $\calB_2$ is minimizing the permanent of the Gram matrix of the
unit representatives.  This is the sharpest form of the reduction: the
configuration enters only through $G$, a positive semidefinite Hermitian matrix
of rank at most $2$ with unit diagonal.

Two consequences follow.  First, the Bombieri product inequality
\cite[Theorem~1.2]{BeauzamyBombieriEnfloMontgomery}, in the normalization of
\cite[Eq.~(1.4.1)]{EtayoHedenmalmOrtega}, gives, for forms of degrees $m$ and $n$,
\[
 \|PQ\|_{\mathrm{BW}}^2
 \geq\binom{m+n}{m}^{-1}\|P\|_{\mathrm{BW}}^2\|Q\|_{\mathrm{BW}}^2,
\]
iterates to $\|\ell_1\cdots\ell_N\|_{\mathrm{BW}}^2\geq1/N!$, so
\begin{equation}\label{eq:theta-N-two-floor}
 \Theta_{N,2}(\bP^1)\geq\frac{e^N}{(N+1)\,N!} .
\end{equation}
Second, \eqref{eq:theta-N-two} is a minimization over a compact
$2N$-real-dimensional manifold of an explicit polynomial, so it is accessible to
exact computation for small $N$ in a way the original integral is not.

The minimizers appear to be balanced, but we are careful about the status of that
statement.

\begin{remark}[Balance of the minimizers]\label{rem:even-exponent-balance}
Write $x_1,\ldots,x_N\in S^2$ for the unit vectors corresponding to
$w_1,\ldots,w_N$ under the Hopf identification, so that $\sum_jx_j$ is, up to a
positive factor, the moment map of the diagonal $\mathrm{PSU}(2)$ action on
$(\bP^1)^N$.  In numerical searches for $2\leq N\leq8$, from many random
starts, the lowest-energy configurations found satisfy
\begin{equation}\label{eq:even-exponent-balance}
 \sum_{j=1}^Nx_j=0 ,
\end{equation}
to the accuracy of the optimization.  We do \emph{not} claim a proof.  The natural
argument, that a $\mathrm{PSU}(2)$-invariant function has balanced critical points,
is false: the configuration with all $N$ points coincident is a critical point of
$\|P_{\Wset}\|_{\mathrm{BW}}^2$ --- indeed its maximum, with value $1$ --- and its moment
$\sum_jx_j$ has norm $N$.  So \eqref{eq:even-exponent-balance} is a property of
minimizers specifically, and any proof must use minimality and not merely
criticality.  Establishing it is the first part of
\Cref{prob:even-exponent-tower}.
\end{remark}

\begin{remark}[Small-cardinality candidate configurations]
\label{rem:even-exponent-platonic}
The following configurations have the indicated exact squared
Bombieri--Weyl norms:
\[
 \begin{array}{c|cccccc}
  N & 2 & 3 & 4 & 5 & 6 & 7\\\hline
  \|\ell_1\cdots\ell_N\|_{\mathrm{BW}}^2
    & \tfrac12 & \tfrac14 & \tfrac19 & \tfrac1{20} & \tfrac1{48}
    & \tfrac1{112}\\[2pt]
  \text{type}
    & \text{antip.} & \text{equil.} & \text{tetra.}
    & \text{bipyr.}_3 & \text{octa.} & \text{bipyr.}_5
 \end{array}
\]
The bipyramids consist of the two poles and a regular equatorial polygon.
The values can be checked directly from the coefficient norm. For a regular
equatorial $r$-gon, the product is, up to phase,
$2^{-r/2}(\zeta_1^r-\zeta_0^r)$. Adding the poles multiplies it by
$\zeta_0\zeta_1$, giving squared norm $2^{1-r}/(r+2)$; this covers the
entries $N=5,6,7$. For a tetrahedron, place one vertex at the north pole
and the other three at height $-1/3$. The two nonzero coefficient magnitudes
are $(2/3)^{3/2}$ and $(1/3)^{3/2}$ in degrees $\zeta_0^3\zeta_1$ and
$\zeta_1^4$, giving $(2/3)^3/4+(1/3)^3=1/9$.
These values give
$\Theta_{N,2}(\bP^1)\leq e^N/\bigl((N+1)d_N\bigr)$ with
$d_N=2,4,9,20,48,112$.  Equality is proved above for $N=2,3$; for
$4\leq N\leq7$ these are candidate upper bounds, also found by numerical
minimization from many random starts.

For $N=8$, consider two squares at heights $c$ and $-c$ on the unit sphere,
rotated by $\pi/4$ relative to each other, with $0\leq c\leq1$.
Let $a^2=(1+c)/2$, $b^2=(1-c)/2$, and $t=c^2$.  Unit representatives are
$(a,bi^k)$ and $(b,ae^{i\pi/4}i^k)$, $0\leq k\leq3$, so, up to a phase,
\[
 P_{\Wset}=(a^4\zeta_1^4-b^4\zeta_0^4)(b^4\zeta_1^4+a^4\zeta_0^4),
 \qquad
 f(t):=\|P_{\Wset}\|_{\mathrm{BW}}^2
 =\frac{(1-t)^4}{128}+\frac{t(1+t)^2}{280}.
\]
Since $f''(t)=3(1-t)^2/32+(4+6t)/280>0$ on $[0,1]$ and
$f'(0)<0<f'(1)$, the unique minimum within this family occurs at the root
$t_*\in(0,1)$ of
\[
 35t_*^3-93t_*^2+121t_*-31=0.
\]
It gives $\sqrt{t_*}=0.5737201649\ldots$ and
$f(t_*)=0.003659058648\ldots$, hence the rigorous upper bound
$\Theta_{8,2}(\bP^1)\leq e^8f(t_*)/9$.
The regular square antiprism has $c^2=(1+2\sqrt2)^{-1}$, obtained by equating
its square and sloping edge lengths; thus the optimizing squares are closer
to the poles than in the regular antiprism.  The value $c=1/\sqrt3$ instead
gives $f(1/3)=83/22680$.  Global optimality of the configuration at $t_*$
among all eight-point configurations is not asserted.

The candidate energies in this range are close to $1.2$, consistently with
the two-sided bounds of \Cref{thm:theta-uniform-noncollapse}.
\end{remark}

\subsection{Problems suggested by the exact model}
\label{subsec:sphere-problems}

The complete $N=2$ computation is instructive about what a general answer
would have to look like.  Three features are worth isolating.  First, the two-point
energy is a ratio of gamma factors times a hypergeometric function of a
\emph{single} geometric variable, the squared chordal distance to the antipode.
The minimizer is unique modulo the isometry group and the labelling.  And the
Hessian is nondegenerate transverse to the group orbit, so the minimum is
Morse--Bott.  Each of these suggests a question at the next cardinality.

\begin{problem}[The three-point round sphere]\label{prob:sphere-three-point}
Determine $\Theta_{3,\beta}(\bP^1)$ and classify its minimizers for
$\beta\neq2$.  The case $\beta=2$ is settled by
\Cref{thm:sphere-three-point-beta2}, where the minimizers are the equilateral
great-circle triangles; numerical minimization supports the same configuration at
every other exponent tested.  A proof would presumably
reduce $\calB_{\beta}$ to the three pairwise spinor overlaps
$s_{ij}=1-d_{\rm ch}(w_i,w_j)^2$, giving an Appell or Lauricella series in
$(s_{12},s_{13},s_{23})$, and then exhibit a positivity or Schur-convexity
structure forcing the equilateral Gram matrix.  Symmetry of the candidate is not
by itself an argument: by \Cref{thm:sphere-two-point-exact} the $N=2$ minimizer
is also the symmetric configuration, but that conclusion came from strict
monotonicity in one variable, and no such reduction is available for $N=3$.
\end{problem}

\begin{problem}[Nondegeneracy for every $N$]\label{prob:sphere-morse-bott}
Is every minimizer of $\calB_{\beta}(\,\cdot\,;\bP^1)$ Morse--Bott transverse to its
$\mathrm{PSU}(2)$ orbit, for every $N$ and every $\beta>0$?
\Cref{cor:sphere-two-point-morse-bott} proves this for $N=2$ with the explicit
normal eigenvalue $\beta^2\Theta_{2,\beta}/4$. The case $N=3$, $\beta=2$
is covered by \Cref{cor:sphere-three-point-morse-bott}. A uniform answer would give
local uniqueness modulo symmetry.  Since the minimizing set is compact and would
then be covered by finitely many such neighbourhoods, it would follow that the
minimizers form a finite union of $\mathrm{PSU}(2)$ orbits.
\end{problem}

\begin{problem}[Exponent-independence of the minimizer]
\label{prob:sphere-beta-independence}
For $N=2$ the minimizing configuration is the antipodal pair for \emph{every}
$\beta>0$: the exponent affects the optimal constant but not the optimal
configuration.  For which $\Xsurf$ and $N$ does this persist?  Equivalently, when is
the argmin of $\calB_{\beta}(\,\cdot\,;\Xsurf)$ independent of $\beta$?  By
\Cref{prop:theta-minimizer-endpoint-selection} such independence would force the
same configuration to minimize both the quadratic Green energy $\mathcal E_N$
and the pointwise endpoint $\|Y_{\Wset}\|_{\infty}$, so a natural first step is to decide
whether those two variational problems have a common minimizer.  A negative
answer for some $(\Xsurf,N)$ would show that the minimizing configurations depend
on the exponent.  Such dependence need not cause a failure of differentiability
of $\mathcal F_N$: by \Cref{prop:theta-Gibbs-envelope}, nondifferentiability
requires minimizers with different Gibbs entropies at the same exponent.
\end{problem}

\section{Minimizing configurations}\label{sec:minimizers}

\subsection{Existence, simplicity, and balance}

For each fixed $N$ and $\beta$ the infimum in \eqref{eq:B-def} is attained.  We
prove that minimizers exist, that they contain no collisions, that they inherit
the optimal $W_1$ scale from \Cref{sec:quantization}, and that they satisfy an
intrinsic force-balance equation.  What remains open is genuinely microscopic
and is stated in
\Cref{prob:minimizer-characterization,prob:minimizer-separation}.

\begin{problem}[Characterization of minimizers]
\label{prob:minimizer-characterization}
Classify configurations which minimize $\calB_{\beta}(\Wset;\Xsurf)$, determine whether
they are unique modulo symmetries and permutations, and compare them with
Fekete points, zeros of extremal sections, and energy-minimizing divisors.
\end{problem}

Existence, absence of collisions, optimal-scale equidistribution, and the
first-order balance law can all be settled here.  The unresolved content of
the problem is therefore genuinely microscopic.  The fixed-$N$ endpoint
selection in \Cref{prop:theta-minimizer-endpoint-selection} identifies the
limiting variational problems as $\beta\downarrow0$ and
$\beta\to\infty$, but it neither classifies their minimizers nor addresses
the simultaneous large-$N$ limit.

\begin{proposition}[Existence of a minimizing configuration]
\label{prop:minimizer-exists}
For every fixed $N\geq1$ and $\beta>0$, the infimum defining
$\Theta_{N,\beta}(\Xsurf)$ is attained by some $\Wset\in \Xsurf^N$.
\end{proposition}

\begin{proof}
The Green kernel is continuous off the diagonal and is bounded above on
$\Xsurf\times \Xsurf$; its diagonal logarithmic singularity tends to $-\infty$.
By \Cref{lem:green-joint-local}, in joint local coordinates
$U_{\Xsurf}(z,w)=\log|z-w|+R(z,w)$ with $R$ smooth.  Hence
$e^{\beta U_{\Xsurf}(z,w)}=|z-w|^\beta e^{\beta R(z,w)}$ extends continuously
to the diagonal by zero.  Products of these extended factors are therefore
continuous in $(z,\Wset)$.
If $\Wset^{(k)}\to \Wset$ in $\Xsurf^N$, then the integrands defining
$\calB_{\beta}(\Wset^{(k)};\Xsurf)$ converge pointwise for every $z\in \Xsurf$.  If
$M_{\Xsurf}:=\sup_{\Xsurf\times \Xsurf}U_{\Xsurf}<\infty$, then every
integrand is bounded above by the integrable constant
$\exp(\beta N M_{\Xsurf})$, uniformly for all configurations in $\Xsurf^N$.  Dominated
convergence
therefore shows that $\calB_{\beta}(\,\cdot\,;\Xsurf)$ is continuous on the compact
space $\Xsurf^N$, so it attains its minimum.
\end{proof}

\begin{corollary}[Minimizers lie in a fixed energy class]
\label{cor:minimizer-in-class}
Fix $\beta>0$ and, for each $N$, let $\Wset_N^*$ attain $\Theta_{N,\beta}(\Xsurf)$.  Then
\[
 \calB_{\beta}(\Wset_N^*;\Xsurf)=\Theta_{N,\beta}(\Xsurf)\leq K_{\beta}:=e^{\beta C_{\Xsurf,\pmeas}}
\]
for every $N$, where $C_{\Xsurf,\pmeas}$ is the constant of
\Cref{thm:packing-input}.  Thus $(\Wset_N^*)_N\in\calB_{\beta}(K_{\beta})$, and
every result of
\Cref{sec:discrepancy,sec:quantization,sec:mesoscopic} applies to minimizing
configurations.
\end{corollary}

\begin{proof}
The configurations of \Cref{thm:packing-input} are admissible competitors, so
$\Theta_{N,\beta}(\Xsurf)\leq \calB_{\beta}(\Wset_N;\Xsurf)\leq e^{\beta C_{\Xsurf,\pmeas}}$, and the
minimum is attained by \Cref{prop:minimizer-exists}.  The bound is independent
of $N$, which is \Cref{def:energy-class}.
\end{proof}

Every macroscopic conclusion of the preceding sections therefore applies to
minimizers, with no further work.  We record the transport statement explicitly,
since it is the one most often wanted.

\begin{corollary}[Optimal $W_1$ scale for minimizers]
\label{cor:minimizer-equidistribution}
For each $N$, choose a minimizing configuration $\Wset_N^*$ for
$\Theta_{N,\beta}(\Xsurf)$ and let $\nu_N^*$ be its empirical measure.  Then
\begin{equation}\label{eq:minimizer-Wasserstein}
 c_{\Xsurf}N^{-1/2}\leq W_1(\nu_N^*,\pmeas)\leq C_{\Xsurf,\pmeas,\beta}N^{-1/2}.
\end{equation}
They also satisfy the coarse no-hole estimate
\begin{equation}\label{eq:minimizer-coarse-covering}
 \mathcal R(\Wset_N^*)\leq C_{\Xsurf,\pmeas,\beta}N^{-1/4}.
\end{equation}
In particular, every sequence of minimizers converges to $\pmeas$ at the
optimal quantization scale.
\end{corollary}

\begin{proof}
By \Cref{cor:minimizer-in-class},
$\calB_{\beta}(\Wset_N^*;\Xsurf)=\Theta_{N,\beta}(\Xsurf)\leq K_{\beta}$ with $K_{\beta}$ independent
of $N$, so the minimizing configurations form a family in
$\calB_{\beta}(K_{\beta})$.  The upper bound therefore follows from
\Cref{cor:Wasserstein}, and the lower bound is
\Cref{cor:Wasserstein-lower}.  The covering estimate follows from
\Cref{cor:energy-covering-radius} applied to the same family.
\end{proof}

\begin{theorem}[Collisions are not local minima]
\label{thm:minimizer-simple}
No configuration containing a collision is a local minimum of
$\calB_{\beta}(\,\cdot\,;\Xsurf)$ on $\Xsurf^N$.  In particular, every minimizing
configuration for $\Theta_{N,\beta}(\Xsurf)$ consists of $N$ distinct points.
\end{theorem}

\begin{proof}
Suppose that $w_1=w_2=w$.  Choose a holomorphic coordinate $\xi$ centered
at $w$ and, for small $a\in\C$, replace these two points by
$w_+(a)=\xi^{-1}(a)$ and $w_-(a)=\xi^{-1}(-a)$, keeping all other points
fixed.  Let $F(a)$ be the resulting value of $\calB_{\beta}$.  We prove an actual
descent statement by averaging $F$ over the parameter circle, rather than
passing Hessians through a singular regularization.

Write $\diff\pmeas(x)=m_{\xi}(x)dA(x)$ in the coordinate, where
$dA=(i/2)dx\wedge d\bar x$ and $m_{\xi}(0)>0$, and write
\[
 \mathcal M_tF:=\frac1{2\pi}\int_0^{2\pi}F(te^{i\theta})\,d\theta.
\]
We claim that
\begin{equation}\label{eq:collision-circular-variation}
 \mathcal M_tF-F(0)
 =-\pi\beta m_{\xi}(0)F(0)t^2+o(t^2).
\end{equation}
Since $F(0)=\calB_{\beta}(\Wset;\Xsurf)\geq1$, this implies
$\mathcal M_tF<F(0)$ for small $t>0$, and hence at least one point of the
circle $|a|=t$ has $F(a)<F(0)$.

To prove the claim, apply \Cref{lem:green-joint-local} and shrink the
coordinate disk so that the local Green expansion has the form
\[
 U_{\Xsurf}\bigl(\xi^{-1}(x),\xi^{-1}(q)\bigr)
 =\log|x-q|+R(x,q),
\]
with $R$ smooth on the bidisk.  The integrand defining $F(a)$ has there the
form
\begin{equation}\label{eq:collision-local-integrand}
 G(a,x)=Q(x)e^{\beta(R(x,a)+R(x,-a))}|x^2-a^2|^\beta,
\end{equation}
where
\[
 Q(x)=\exp\!\left(\beta\sum_{j=3}^N
 U_{\Xsurf}(\xi^{-1}(x),w_j)\right)
\]
is bounded.  Indeed, every factor $e^{\beta U_{\Xsurf}(\,\cdot\,,w_j)}$ extends
continuously by zero at its pole.  This observation also covers collisions
of multiplicity greater than two.

Choose
\[
 \frac1{1+\beta}<\alpha<1,
 \qquad \rho_t=t^\alpha.
\]
Because $t=o(\rho_t)$, for $|a|=t$ and $|x|\leq\rho_t$ both $G(a,x)$ and
$G(0,x)$ are at most $C\rho_t^{2\beta}$.  The inner disk therefore
contributes at most
\begin{equation}\label{eq:collision-inner-error}
 C\rho_t^{2\beta+2}=o(t^2)
\end{equation}
to the absolute difference between $\mathcal M_tF$ and $F(0)$.

On $|x|\geq\rho_t$, the dependence on $a$ is smooth.  Circular Taylor
expansion gives
\[
 \frac1{2\pi}\int_0^{2\pi}G(te^{i\theta},x)\,d\theta
 =G(0,x)+\frac{t^2}{4}\Delta_aG(0,x)+\mathcal R_t(x).
\]
On the coordinate disk, differentiation of
\eqref{eq:collision-local-integrand} yields
\[
 |\mathcal R_t(x)|\leq
 Ct^4\bigl(1+|x|^{2\beta-4}\bigr),
\]
while outside the disk the remainder is $O(t^4)$.  Hence
\[
 \int_{|x|\geq\rho_t}|\mathcal R_t(x)|\,\diff\pmeas(x)
 \leq
 \begin{cases}
  Ct^4\rho_t^{2\beta-2},&0<\beta<1,\\
  Ct^4|\log\rho_t|,&\beta=1,\\
  Ct^4,&\beta>1,
 \end{cases}
 =o(t^2).
\]

If $Q(x)=0$ because $x$ is one of the fixed nodes, then
$G(a,x)\equiv0$ as a function of $a$, so every parameter identity below is
trivially valid there.  We may therefore exclude this finite set and write
$G=e^\Phi$ where $G>0$.  The resulting Laplacian identity holds for almost
every $x$, which is all that is needed after integration.

For $x\neq0$, the opposite parameter velocities give
$d_a\Phi(0,x)=0$, where $G=e^\Phi$.  More explicitly, the Green equation in
the pole variable gives
\[
 \left.dd^c_aU_{\Xsurf}\bigl(\xi^{-1}(x),\xi^{-1}(a)\bigr)\right|_{a=0}
 =-m_{\xi}(0)dA(a),
\]
because $x\neq0$.  The term with pole $\xi^{-1}(-a)$ is the pullback by the
holomorphic map $a\mapsto-a$, which gives the same value at the origin.
Therefore
\[
 dd^c_a\Phi(0,x)=-2\beta m_{\xi}(0)dA(a).
\]
Since $d_a\Phi(0,x)=0$, the chain rule gives
$dd^c_aG(0,x)=G(0,x)dd^c_a\Phi(0,x)$.  Here $\Delta_a=\partial_{\operatorname{Re}a}^2+
\partial_{\operatorname{Im}a}^2$ is the Euclidean parameter Laplacian. Thus
$dd^c_af=(\Delta_af/2\pi)dA(a)$ in our convention, and
\begin{equation}\label{eq:collision-far-laplacian}
 \Delta_aG(0,x)=-4\pi\beta m_{\xi}(0)G(0,x).
\end{equation}
The same identity holds off the coordinate disk, where the dependence on
$a$ is smooth, because it used only the pole-variable Green equation and the
opposite velocities.  Moreover,
$\int_{|x|<\rho_t}G(0,x)\diff\pmeas(x)=O(\rho_t^{2\beta+2})=o(1)$.
Integrating the circular Taylor expansion, using
\eqref{eq:collision-inner-error} and
\eqref{eq:collision-far-laplacian}, and adding the
smooth complement of the coordinate disk proves
\eqref{eq:collision-circular-variation}.  Thus the original colliding
configuration is not a local minimum.
\end{proof}

We may therefore differentiate at every minimizing configuration.
Variation with respect to $w_j$ gives the intrinsic balance condition
\begin{equation}\label{eq:Euler-Lagrange-minimizer}
 \int_{\Xsurf} \nabla_{w_j}U_{\Xsurf}(z,w_j)
 \exp\!\left(\beta\sum_{i=1}^NU_{\Xsurf}(z,w_i)\right)\diff\pmeas(z)=0,
 \qquad j=1,\ldots,N.
\end{equation}
Indeed, for a tangent vector $v\in T_{w_j}\Xsurf$, differentiating along a curve
with velocity $v$ gives
\[
 \frac{d}{dt}\calB_{\beta}(w_1,\ldots,w_j(t),\ldots,w_N;\Xsurf)\bigg|_{t=0}
 =\beta\left\langle
 \int_{\Xsurf}\nabla_{w_j}U_{\Xsurf}(z,w_j)e^{\beta\sum_iU_{\Xsurf}(z,w_i)}\diff\pmeas(z),v
 \right\rangle.
\]
At a minimizer this vanishes for every $v$; since $\beta>0$, the vector
inside the pairing must be zero, which is
\eqref{eq:Euler-Lagrange-minimizer}.
The differentiation is legitimate because the local factor
$e^{\beta U_{\Xsurf}(z,w_j)}\asymp d(z,w_j)^\beta$ makes the singular first
variation locally integrable for every $\beta>0$.  To justify the passage
from difference quotients to that first variation, choose a coordinate disk
containing $w_j$ but no other node.  Along a curve $w_j(t)$ the local
integrand has the form
\[
 A(t,x)|x-q(t)|^\beta,
 \qquad q(0)=0,\qquad |q(t)|\leq C|t|,
\]
where $A$ and $\partial_tA$ are uniformly bounded; this uses
\Cref{lem:green-joint-local}.  If $|x|>2C|t|$, the mean-value theorem bounds
the difference quotient by
\[
 C\bigl(1+|x|^{\beta-1}\bigr).
\]
If $|x|\leq2C|t|$ and $0<\beta<1$, the $\beta$-H\"older continuity of
$r\mapsto r^\beta$ bounds it, for almost every $x$, by
$C|t|^{\beta-1}\leq C'|x|^{\beta-1}$; here the inequality reverses in the
desired direction because $\beta-1<0$ and $|x|\leq2C|t|$.  For
$\beta\geq1$ the same difference
quotients are locally bounded.  Thus in every case they admit the
$t$-independent local majorant
$C(1+|x|^{\beta-1})$, which belongs to $L^1(\C)$ near the origin.  Away
from the disk all parameter dependence is smooth and uniformly bounded.
Dominated convergence now proves the displayed first-variation formula.

Equation \eqref{eq:Euler-Lagrange-minimizer} resembles a weighted force-balance
law, and it provides a starting point for comparison with Fekete points and
$\beta$-ensembles.  Whether the system is nondegenerate modulo permutations and
geometric symmetries is the remaining local question.

\begin{problem}[Quantitative separation of minimizers]
\label{prob:minimizer-separation}
Does there exist $c_{\Xsurf,\pmeas,\beta}>0$ such that, for every $N\geq2$, every minimizing
configuration satisfies
\[
 \min_{i\neq j}d(w_i,w_j)\geq c_{\Xsurf,\pmeas,\beta}N^{-1/2}?
\]
Such a bound would upgrade the collision theorem to the natural microscopic
scale and supply the separation condition needed for a planar Delone
blow-up analysis of minimizers.  A matching $O(N^{-1/2})$ covering-radius
estimate would remain open; \eqref{eq:minimizer-coarse-covering} currently
gives only $O(N^{-1/4})$.  Separation does not follow from the optimal
$W_1$ estimate alone, because transport distance is insensitive to a small
number of tightly clustered pairs.
\end{problem}

There is a nonquantitative fixed-$N$ version.  For fixed $N\geq2$ and $\beta$,
the set of minimizers is a closed subset of the compact space $\Xsurf^N$, hence
compact.  By \Cref{thm:minimizer-simple} it is disjoint from every diagonal
$\{w_i=w_j\}$.  The continuous function
$\Wset\mapsto\min_{i\neq j}d(w_i,w_j)$ therefore has a strictly positive minimum
on the minimizer set.  The open content of
\Cref{prob:minimizer-separation} is precisely to control this minimum by
$cN^{-1/2}$ with $c$ independent of $N$.

\section{Microscopic structure: what remains open}\label{sec:microscopic}

\subsection{Microscopic limitations and open questions}

The results of \Cref{sec:quantization,sec:mesoscopic} settle the macroscopic
and mesoscopic behavior of any family in $\calB_{\beta}(K)$, and
\Cref{sec:minimizers} settles existence, simplicity, and the first-order
balance law for minimizers. These consequences of the exponential moment
do not determine the arrangement of the points at the natural scale
$N^{-1/2}$; a sharp limit for $\Theta_{N,\beta}$ requires further information
at that scale.  We close by
describing the obstruction precisely, so that the open problems above are not
mistaken for routine gaps.

Weak convergence $\nu_{\Wset_N}\to\pmeas$, even at the optimal $W_1$ rate, is
insensitive to the microscopic pattern: by \Cref{rem:sharp-and-nonsharp}\,(iv)
it does not even exclude clusters.  A limiting constant should instead be
governed by a local cell problem in the plane.  The natural formulation is a
blow-up.  Fix base points $x_N\in \Xsurf$, express $\Wset_N$ in coordinates rescaled by
$\sqrt N$ at $x_N$, and ask for subsequential limits of the resulting planar
point sets together with their renormalized logarithmic potentials.
Passing the energy to such a limit requires additional local estimates and
control of the far-field contribution.  Indeed, for a fixed rescaled disk,
the change of variables $z=\xi^{-1}(u/\sqrt N)$ gives only
\[
 \int_{|u|<R}e^{\beta V_{\Wset_N}(\xi^{-1}(u/\sqrt N))}\,dA(u)
 \leq C_R N K
\]
from \eqref{eq:energy-hypothesis}.  A uniform local bound after a specified
additive renormalization, and convergence of the resulting potentials, must
therefore be established before a limiting energy can be justified.

A uniform separation lower bound and covering upper bound at scale
$N^{-1/2}$ would provide a Delone property and nondegenerate geometric
compactness in the rescaled metric.  The covering estimate proved here is only
$O(N^{-1/4})$.  Separation does not follow from the energy bound: duplicating
every point of a configuration supplied by \Cref{thm:packing-input} gives
$V_{\Wset_{2M}}=2V_{\Wset_M}$, hence uniformly bounded energy and collisions.

For the holomorphic family of \cite{PaperA} this Delone geometry is
available --- the zero divisors are simple, separated at scale $N^{-1/2}$, and
covering at that scale, with uniform Voronoi geometry --- so that family does
have Delone blow-up limits, and the exponential energy decomposes exactly over
its Voronoi cells.  That is a statement about one particular family, not about
minimizers, and it does not identify a limiting cell: the local norm model
available there gives comparability of the cell energies rather than their
convergence.  A genuine local $\Gamma$-limit would require four further
ingredients.  One needs
convergence of the rescaled representatives, control of the Green contribution from
outside the rescaled window, a recovery-sequence construction patching planar cells
across varying tangent charts, and a comparison of the resulting cell energy with
the infimum defining $\Theta_{N,\beta}$.  Only after
all of those steps would a specific limiting pattern, such as a triangular
lattice, be justified; symmetry of a candidate configuration is not by itself
evidence for it.

For minimizers the situation is more basic still: the natural-scale separation
asked for in \Cref{prob:minimizer-separation} is precisely the separation half
of the Delone condition, and the matching $O(N^{-1/2})$ covering bound is likewise
open.  This is why \Cref{prob:optimal-constants} and
\Cref{prob:minimizer-characterization} are stated as they are.  The
macroscopic theory of this paper is complete; the microscopic theory requires
additional local information.

\section*{Acknowledgements}
Hedenmalm acknowledges support from the Ministry of Science and Higher
Education of the Russian Federation under agreement 075-15-2025-013, as well
as the hospitality of BIMSA.
S.-Y.~Xie acknowledges partial support from the National Key R\&D Program
of China (grant nos.~2023YFA1010500 and~2021YFA1003100), the National Natural
Science Foundation of China (grant nos.~12288201 and~12471081), and the
Xiaomi Young Talents Program.

\section*{Declaration on the use of artificial intelligence}
The authors formulated the problems and supplied the mathematical ideas.
GPT-5.6-sol was used to assist mathematical exploration and manuscript
preparation. The authors checked all mathematical arguments and take full
responsibility for the contents.

\end{document}